\documentclass[a4paper,10pt]{amsart}
\usepackage{amsmath,amsthm,amssymb,amsfonts,enumerate,color,bbm}

\newcommand{\set}[1]{\left\{#1\right\}}
\newcommand{\Real}{\mathbb{R}}
\renewcommand{\a}{\alpha}
\renewcommand{\b}{\beta}

\renewcommand{\t}{\theta}
\renewcommand{\v}{\varphi}

\renewcommand{\P}{\mathcal{P}}
\newcommand{\J}{\mathcal{J}}
\newcommand{\K}{\mathcal{K}}
\renewcommand{\S}{\mathbb{S}}

\newcommand{\CP}{P_l(\mathbb{C})}
\newcommand{\HP}{P_l(\mathbb{H})}
\newcommand{\Ca}{P_2(\mathbb{C}\mathbbm{a}\mathbbm{y})}
\newcommand{\B}{\mathbb{B}}
\newcommand{\dive}{\operatorname{div}}

\newtheorem{thm}{Theorem}[section]

\newtheorem{lem}[thm]{Lemma}

\theoremstyle{definition}

\allowdisplaybreaks

\numberwithin{equation}{section}

\author[\'O. Ciaurri, L. Roncal, and P. R. Stinga]{\'Oscar Ciaurri \and Luz Roncal \and Pablo Ra\'ul Stinga}
\address{Departamento de Matem\'aticas y Computaci\'on\\
         Universidad de La Rioja\\
         26004 Logro\~no, Spain}
\email{oscar.ciaurri@unirioja.es, luz.roncal@unirioja.es}

\address{Department of Mathematics\\
	The University of Texas at Austin\\
	1 University Station C1200\\
	78712-1202 Austin, TX\\
	United States of America}
\email{stinga@math.utexas.edu}

\thanks{Research partially supported by grants MTM2012-36732-C03-02 and MTM2011-28149-C02-01 from Spanish Government}



\begin{document}

\title[Riesz transforms on Riemannian symmetric spaces]{Riesz transforms on compact \\ Riemannian symmetric spaces of rank one}

\begin{abstract}
In this paper we prove mixed norm estimates for Riesz transforms related to Laplace--Beltrami operators on compact Riemannian symmetric spaces of rank one. These operators are closely related to the Riesz transforms
for Jacobi polynomials expansions. The key point is to obtain
sharp estimates for the kernel of the Jacobi--Riesz transforms with uniform control on
the parameters, together with an adaptation of Rubio de Francia's
extrapolation theorem. The latter results are of independent interest.
\end{abstract}

\maketitle

\section{Introduction}

Let $M$ be a Riemannian symmetric space of rank
one\footnote{These are the ``model'' spaces of Riemannian geometry, cf. Isaac Chavel,
\textit{Riemannian
Symmetric Spaces of Rank One},
Lecture Notes in Pure and Applied Mathematics, Vol. 5.
Marcel Dekker, Inc., New York, 1972.} and compact type.
The Laplace--Beltrami operator $\tilde{\Delta}_M$ on $M$ is given by $\tilde{\Delta}_M=\dive_M\nabla_M$,
where $\nabla_M$ and $\dive_M$ are the Riemannian gradient and divergence, respectively.
Suppose now that $u$ is a solution to the nonlocal problem
$$(-\Delta_M)^{1/2}u=f,\quad\hbox{in}~M,$$
where $(-\Delta_M)^{1/2}$ is the square root of $-\Delta_M=-\tilde{\Delta}_M+\lambda_M$,
and $\lambda_M$ is a positive constant depending on $M$ (see \eqref{lambdas} below).
The operator $(-\Delta_M)^{1/2}$ can be thought as a \textit{first order} (nonlocal) differential operator.
Thus, when $f$ is in some functional space $X$ we may expect
 the gradient of $u$ to be in the same space $X$,
that is,
$$|\nabla_M u|=|\nabla_M(-\Delta_M)^{-1/2}f|\in X.$$
The operator
$$R_Mf:=|\nabla_M(-\Delta_M)^{-1/2}f|,$$
is the Riesz transform on $M$. The consideration above
 would say that the Riesz transform is a bounded operator from $X$ to $X$.
 In this paper we choose as $X$ the mixed norm spaces $L^p(L^2(M))$
 that are defined in terms of polar coordinates on $M$.
Here is our main result.

\begin{thm}\label{Thm:Riesz Riemannian}
Let $M$ be a compact Riemannian symmetric space of rank one. Then the Riesz transform $R_M$
is a bounded operator from $L^p(L^2(M))$ into itself, for all $1<p<\infty$.
\end{thm}

It is known that the Riesz transforms based on the Laplace--Beltrami operator
$\nabla_M(-\tilde{\Delta}_M)^{-1/2}$ are bounded in $L^p(M)$, $1<p<\infty$, even if $M$ is of
non-compact type, see \cite[Section~6]{Strichartz}. Nothing had been done so far
for singular integrals in the mixed norm context.
In this regard, Theorem \ref{Thm:Riesz Riemannian} is completely original.
Moreover, the ideas and techniques developed here open the way to the study of many important singular
integral operators related to the Laplace--Beltrami operator in mixed norm spaces, like
maximal operators, higher order Riesz transforms and Littlewood--Paley square functions.
Applications to Sobolev spaces on $M$ and further development of the theory will appear elsewhere.

H.-C. Wang showed in \cite{Wang} that there are just five examples of $M$: the sphere, the real, complex and
quaternionic projective spaces and the Cayley plane.
Each of these manifolds admit appropriate polar coordinates
$(\t,x')\in (0,\pi)\times\S_M$, where $\S_M$ is a unit sphere whose dimension
depends on $M$.
This allows us to define the mixed norm spaces $L^p(L^2(M))$, see \cite{Ciaurri-Roncal-Stinga}.
It turns out that in these coordinates,
$$-\Delta_M=\mathcal{J}^{\a,\b}-\rho_M(\t)\tilde{\Delta}_{\S_M},$$
where, for given parameters $\a,\b>-1$ that depend on $M$, $\mathcal{J}^{\a,\b}$
is the trigonometric Jacobi polynomials differential operator (see \eqref{eq:ecDiferencial Jacobi})
and $\rho_M(\t)$ is an explicit nonnegative function of $\t$.
Thus the eigenspaces $\mathcal{H}_n(M)$
 of $-\Delta_M$ can eventually be written in terms of products of spherical harmonics
 and trigonometric Jacobi polynomials whose type parameters depend
 on $j$, where $j$ varies in a range depending on $n$ and the dimension of $M$.
 Moreover, because of such an underlying structure, the Riesz transform on $M$ in the spaces $L^p(L^2(M))$
 is related to the Jacobi--Riesz transform
 $$\mathcal{R}^{\a,\b}:=\partial_\t(\J^{\a,\b})^{-1/2},$$
 and an operator given by a negative power of the Jacobi operator
 $\mathcal{T}_M^{\a,\b}:=\sqrt{\rho_M(\t)}(\J^{\a,\b})^{-1/2}$,
 see Section \ref{Section:Riemannian} for more details.
Now the key point of the paper is to exploit these connections.
Roughly speaking, because of the structure of the operators and the estimates we are looking for\footnote{See
Section \ref{Section:Riemannian} for the detailed explanation.}, we need
suitable weighted vector-valued extensions for the operators $\mathcal{R}^{\a,\b}$ and $\mathcal{T}_M^{\a,\b}$.
These are contained in Theorems \ref{Cor:Lp} and \ref{Thm:Lp angular}.
We achieve both extensions with a clever use of a suitable adaptation of the extrapolation
theorem by Rubio de Francia, see Theorem \ref{th:extra-RF}.
In order to apply such an extrapolation result, we need to prove that
$\mathcal{R}^{\a,\b}$ and $\mathcal{T}_M^{\a,\b}$ are Calder\'on--Zygmund operators
in the sense of suitable spaces of homogeneous type, and show appropriate estimates for the corresponding kernels, see Theorems \ref{Thm:Jacobi-Riesz kernel}
and \ref{Angular Riesz kernel}.
In this way, by the general theory, they will satisfy
weighted inequalities. But we have to be careful in the estimates for the kernels,
since they need to be uniform on the parameters $\a,\b$.
These estimates are sharper than previous ones
already obtained in the Jacobi context (see \cite{Ciaurri-Roncal-Stinga, Nowak-Sjogren Calderon})
and of independent interest.

The paper is organized as follows.
In Section \ref{Section:Riemannian} we present the proof of Theorem \ref{Thm:Riesz Riemannian}
under the assumption that Theorems \ref{Cor:Lp} and \ref{Thm:Lp angular} are true.
Then the rest of the paper is devoted to complete the proof of the vector-valued extensions.
Section \ref{proofs vector-valued}
contains the adaptation of Rubio de Francia's
extrapolation theorem and also shows how to use it to derive Theorems \ref{Cor:Lp}
and \ref{Thm:Lp angular}.
In Section \ref{Section:Kernels} we prove the
kernel estimates for $\mathcal{R}^{\a,\b}$ and $\mathcal{T}_M^{\a,\b}$.
Since in this last section we must get uniform estimates on $\a$ and $\b$,
the computations become rather cumbersome,
so we will try to keep them to a minimum.

\section{The Riesz transform on $M$}\label{Section:Riemannian}

In this section we give the precise definitions of the mixed norm spaces $L^p(L^2(M))$
and the Riesz transform $R_M$.
We will also prove Theorem \ref{Thm:Riesz Riemannian} (case by case)
by using Theorems \ref{Cor:Lp} and \ref{Thm:Lp angular}.

As we said above, the manifolds $M$ we are considering here
are completely classified as
\begin{enumerate}
    \item the sphere $\S^d\subset\Real^{d+1}$, $d\geq1$;
    \item the real projective space $P_d(\Real)$, $d\geq2$;
    \item the complex projective space $\CP$, $l\geq2$;
    \item the quaternionic projective space $\HP$, $l\geq2$;
    \item the Cayley plane $\Ca$.
\end{enumerate}
We take (see \cite{Ciaurri-Roncal-Stinga, Sherman})
\begin{equation}\label{lambdas}
\lambda_{\S^d}=\lambda_{P_d(\Real)}=(\tfrac{d-1}{2})^2,\quad
\lambda_{\CP}=\lambda_{\HP}=\lambda_{\Ca}=(\tfrac{m+d}{2})^2.
\end{equation}
Here $d=2,4,8$, $m=l-2,2l-3,3$, for $\CP$, $\HP$ and $\Ca$, respectively.
Besides, we will assume that $d\ge2$ in the case of $\S^d$.

\subsection{Preliminaries on Jacobi expansions}

In this subsection we quickly introduce some necessary facts about Jacobi expansions and the definitions of the
related operators that will arise in the proof of Theorem \ref{Thm:Riesz Riemannian}.

The standard Jacobi polynomials of degree $n\geq0$ and type $\a,\b>-1$ are given by
\begin{equation}\label{eq1}
P_n^{(\a,\b)}(x)=(1-x)^{-\a}(1+x)^{-\b}\frac{(-1)^n}{2^nn!}\left(\frac{d}{dx}\right)^n\set{(1-x)^{n+\a}(1+x)^{n+\b}},\quad x\in(-1,1),
\end{equation}
see \cite[(4.3.1)]{Szego}. These form an orthogonal basis of $L^2((-1,1),(1-x)^\a(1+x)^\b dx)$. After making the change of variable $x=\cos\theta$, we obtain the normalized Jacobi trigonometric polynomials
\begin{equation}\label{trig pol}
\mathcal{P}_n^{(\a,\b)}(\theta)=d_n^{\a,\b}P_n^{(\a,\b)}(\cos\theta),
\end{equation}
where the normalizing factor is
\begin{equation}\label{que}
\begin{aligned}
    d_n^{\a,\b} &= 2^{\frac{\a+\b+1}{2}}\|P_n^{(\a,\b)}\|_{L^2((-1,1),(1-x)^\a(1+x)^\b dx)}^{-1} \\
     &= \left(\frac{(2n+\a+\b+1)\Gamma(n+1)\Gamma(n+\a+\b+1)}{\Gamma(n+\a+1)\Gamma(n+\b+1)}\right)^{1/2}.
\end{aligned}
\end{equation}
The trigonometric polynomials $\P_n^{(\a,\b)}$ are eigenfunctions of the Jacobi differential operator
\begin{equation}
\label{eq:ecDiferencial Jacobi}
\mathcal{J}^{\a,\b}=-\frac{d^2}{d\theta^2}-\frac{\a-\b+(\a+\b+1) \cos\theta}{\sin\theta}\frac{d}{d\theta}+\left(\frac{\a+\b+1}{2}\right)^2.
\end{equation}
Indeed, we have $\J^{\a,\b}\P_n^{(\a,\b)}=\lambda_n^{\a,\b}\P_n^{(\a,\b)}$, with eigenvalue $\lambda_n^{\a,\b}=\big(n+\tfrac{\a+\b+1}{2}\big)^2$.
Moreover, the system $\{\mathcal{P}_n^{(\a,\b)}\}_{n\ge0}$ forms a complete orthonormal basis of  $L^2(d\mu_{\a,\b}):=L^2((0,\pi),d\mu_{\a,\b}(\t))$,
with
$$d\mu_{\a,\b}(\theta)=(\sin\tfrac{\theta}{2})^{2\a+1}(\cos\tfrac{\theta}{2})^{2\b+1}d\theta.$$
For further information about Jacobi polynomials see \cite[Chapter IV]{Szego}.

\subsubsection{Jacobi--Riesz transforms}\label{definicion jacobiana}

We have the decomposition
$$\mathcal{J}^{\a,\b}=\delta^\ast\delta+\left(\frac{\a+\b+1}{2}\right)^2,$$
where $\delta=\frac{d}{d\t}$,
and $\delta^\ast$ is its formal adjoint in $L^2(d\mu_{\a,\b})$, that is,
$\delta^\ast=-\frac{d}{d\t}-(\a+1/2)\cot\tfrac{\t}{2}+(\b+1/2)\tan\tfrac{\t}{2}$.
The Jacobi--Riesz transform is formally defined as $\mathcal{R}^{\a,\b}:=\delta(\mathcal{J}^{\a,\b})^{-1/2}$.
For a function $f\in L^2(d\mu_{\a,\b})$ we can write $f=\sum_{n=0}^\infty c_n^{\a,\b}(f)\P_n^{(\a,\b)}=
\sum_{n=0}^\infty \langle f,\P_n^{(\a,\b)}\rangle_{L^2(d\mu_{\a,\b})}\P_n^{(\a,\b)}$ in
$L^2(d\mu_{\a,\b})$. Then
\begin{equation}
\label{definicion obvia}
\begin{aligned}
\mathcal{R}^{\a,\b}f(\theta) &= \sum_{n=0}^\infty\frac{1}{\left(n+\frac{\a+\b+1}{2}\right)}\,c_n^{\a,\b}(f)\delta\P_n^{(\a,\b)}(\theta)\\
&=-\frac{1}{2}\sum_{n=0}^\infty\frac{(n(n+\a+\b+1))^{1/2}}{n+\frac{\a+\b+1}{2}}
c_n^{\a,\b}(f)\sin\t\,\P_{n-1}^{(\a+1,\b+1)}(\theta).
\end{aligned}
\end{equation}

\subsubsection{The operator $\mathcal{T}_M^{\a,\b}$}\label{operador auxiliar}

Here we define $\mathcal{T}_M^{\a,\b}$ as in the introduction by taking
$$\rho_{\S^d}(\t)=\frac{1}{(\sin\t)^2},\quad
\rho_{P_d(\Real)}(\t)=\rho_{\CP}(\t)=\rho_{\HP}(\t)=\rho_{\Ca}(\t)=\frac{1}{(\sin\frac{\t}{2})^2}.$$
For $f\in L^2(d\mu_{\a,\b})$,
$$\mathcal{T}_M^{\a,\b}f(\t)=\sqrt{\rho_M(\t)}
\sum_{n=0}^\infty\frac{1}{n+\frac{\a+\b+1}{2}}\,c_n^{\a,\b}(f)\P_n^{(\a,\b)}(\theta).$$

\subsection{The case of the unit sphere $\S^d$}\label{subsection:redonda}

Let $\S^d=\set{x\in\Real^{d+1}:x_1^2+\cdots+x_{d+1}^2=1}$ be the unit sphere in $\Real^{d+1}$, $d\geq2$. We set
$-\Delta_{\S^d}=-\tilde{\Delta}_{\S^d}+\lambda_{\S^d}$, see \eqref{lambdas}.
It is well-known that $L^2(\S^d)=\bigoplus_{n=0}^\infty\mathcal{H}_n(\S^d)$, where $\mathcal{H}_n(\S^d)$ is the set of spherical harmonics of degree $n$ in $d+1$ variables, see \cite[Chapter~3,~Section~C.I]{BGM}. Each $\mathcal{H}_n(\S^d)$ is an eigenspace of $\tilde{\Delta}_{\S^d}$ with eigenvalue $-n(n+d-1)$.
Hence, $-\Delta_{\S^d}(\mathcal{H}_n(\S^d))=(n+\frac{d-1}{2})^2(\mathcal{H}_n(\S^d))$.

We introduce the following coordinates on $\S^d$,
known as \textit{geodesic polar coordinates}. Each point on the sphere can be written as
\begin{equation}\label{geodesicas}
\Phi(\theta,x')=(\cos\theta,x_1'\sin\theta,\ldots,x_d'\sin\theta)\in\S^d,
\end{equation}
for $\theta\in[0,\pi]$ and $x'\in\S^{d-1}$ (see \cite[pp.~69--70]{Chavel},
also \cite[p.~104]{Sherman} after the change $t=\cos\t$). For a function
$f$ on $\S^d$ let us write $F=f\circ\Phi$. We have, see \cite[pp.~56--57]{BGM},
$$\int_{\S^d}f(x)\,dx=\int_0^\pi\int_{\S^{d-1}}F(\theta,x')\,dx'\,(\sin\theta)^{d-1}\,d\theta,$$

By using the coordinates in \eqref{geodesicas}, we can see that
\begin{equation}\label{laplaciano polares}
-\Delta_{\S^d}=-\frac{\partial^2}{\partial\theta^2}-(d-1)\cot\theta\frac{\partial}{\partial\theta}+\lambda_{\S^d} -\frac{1}{\sin^2\theta}\tilde\Delta_{\S^{d-1}}=\J^{\a,\a}-\rho_{\S^d}(\t)\tilde\Delta_{\S^{d-1}},
\end{equation}
where $\tilde{\Delta}_{\S^{d-1}}$ is the spherical part of the Laplacian on $\Real^d$, acting on functions on $\Real^{d+1}$ by holding the first coordinate fixed and differentiating with respect to the remaining variables, see
 \cite[Chapter~II,~Section~5]{Chavel Eigenvalues}. In \eqref{laplaciano polares} we have chosen
  $\alpha=\beta=\frac{d-2}{2}\geq0$ for the operator $\J^{\a,\a}$ of
 \eqref{eq:ecDiferencial Jacobi}.

By using the description of $\mathcal{H}_n(\S^d)$ via spherical harmonics and the coordinates in \eqref{geodesicas}, it is an exercise to see that an orthonormal basis associated to \eqref{laplaciano polares} is given by
$$\varphi_{n,j,k}(x)=\psi_{n,j}(\theta)Y_{j,k}^d(x'),\quad x=\Phi(\theta,x'),~0<\theta<\pi,~x'\in\S^{d-1}.$$
See also \cite[Section~3,~p.~110]{Sherman}. Here, for $n\geq0$ and $j=0,1,\ldots,n$,
\begin{equation}\label{psi}
\psi_{n,j}(\theta)=a_{n,j}(\sin\theta)^jC_{n-j}^{j+\frac{d-1}{2}}(\cos\theta),
\end{equation}
where
\begin{equation*}
C_k^\lambda(x)=\frac{\Gamma(\lambda+1/2)\Gamma(k+2\lambda)}{\Gamma(2\lambda)
\Gamma(k+\lambda+1/2)}P_k^{(\lambda-1/2,\lambda-1/2)}(x),\quad\lambda>-1/2,~x\in(-1,1),
\end{equation*}
is an ultraspherical polynomial, see \cite[(4.7.1)]{Szego}, and $a_{n,j}$ is the normalizing constant
\begin{align*}
    a_{n,j} &= \|(\sin\theta)^jC_{n-j}^{j+\frac{d-1}{2}}(\cos\theta)\|_{L^2((0,\pi),(\sin\theta)^{d-1}\,d\theta)}^{-1} \\
     &= \frac{2^{-(j+\frac{d-1}{2})}\Gamma(2j+d-1)\Gamma(n+d/2)}{\Gamma(j+d/2)\Gamma(n+j+d-1)}\,d_{n-j}^{j+\frac{d-2}{2},j+\frac{d-2}{2}},
\end{align*}
with $d^{\a,\b}_n$ as in \eqref{que}. The functions $Y_{j,k}^d$, $k=1,\ldots,d(j)$, where $d(j)=(2j+d-2)\frac{(j+d-3)!}{j!(d-2)!}$, form an orthonormal basis of spherical harmonics on $\S^{d-1}$ of degree $j\geq0$. The orthogonal projections of $f$ onto the spaces $\mathcal{H}_n(\S^d)$ can be written in coordinates as
\begin{equation}
\label{proyecciones}
\operatorname{Proj}_{\mathcal{H}_n(\S^d)}(f)=\sum_{j=0}^n\sum_{k=1}^{d(j)}\langle f,\varphi_{n,j,k}\rangle_{L^2(\S^d)}\varphi_{n,j,k}.
\end{equation}

We define the mixed norm space $L^p(L^2(\S^d))$, $1\leq p<\infty$,
as the set of functions $f$ on $\S^d$ such that
$$\|f\|_{L^p(L^2(\S^d))}:=\Bigg(\int_0^\pi\Bigg(\int_{\S^{d-1}}|F(\theta,x')|^2 dx'\Bigg)^{p/2}(\sin\theta)^{d-1}\,d\theta\Bigg)^{1/p}<\infty.$$
The spaces $L^p(L^2(\S^d))$ are Banach spaces under this norm.

The Riesz transform on the mixed norm space $L^p(L^2(\S^d))$
are defined in the following way. Let
\begin{equation}\label{definicion integral fraccionaria}
(-\Delta_{\S^d})^{-1/2}f\equiv (-\Delta_{\S^d})^{-1/2}F=
\sum_{n=0}^\infty\frac{1}{n+\frac{d-1}{2}}\operatorname{Proj}_{\mathcal{H}_n(\S^d)}(f),
\end{equation}
(note the abuse of notation in \eqref{definicion integral fraccionaria}) for $f\in L^2(\S^d)$. Next, notice that
$-\Delta_{\S^d}=\delta_{\S^d}^*\delta_{\S^d}+\lambda_{\S^d}$,
with
$$\delta_{\S^d}=x'\frac{\partial }{\partial \t}+\sqrt{\rho_{\S^d}(\t)}\nabla_{\S^{d-1}},$$
and
$$
\delta^*_{\S^d}=-\left(\frac{\partial }{\partial \t}+(d-1)\cot\t-\frac{d-1}{\sin\t}\right)x'+\sqrt{\rho_{\S^d}(\t)}\dive_{\S^{d-1}},
$$
the formal adjoint of $\delta_{\S^d}$ in $L^2(L^2(\S^d))=L^2(\S^d)$. Observe now that, since $\S^d$ is compact,
by \cite[p.~130,~G.IV.2]{BGM}, and \eqref{psi},
\begin{align*}
\int_{\S^d}|\nabla_{\S^d} f(x)|^2\,dx &= \int_{\S^d}f(x)\tilde{\Delta}_{\S^d}f(x)\,dx\\
	&= \int_0^\pi\int_{\S^{d-1}}F(\t,x')\delta^\ast_{\S^d}\delta_{\S^d}F(\t,x')\,dx'(\sin\t)^{d-1}\,d\t \\
	&= \int_0^\pi\int_{\S^{d-1}}|\delta_{\S^d}F(\t,x')|^2\,dx'(\sin\t)^{d-1}\,d\t \\
	&= \int_0^\pi\int_{\S^{d-1}}\left(|\delta F(\t,x')|^2+\rho_{\S^d}(\t)|\nabla_{\S^{d-1}}F(\t,x')|^2\right)
	\,dx'(\sin\t)^{d-1}\,d\t,
\end{align*}
with $\delta=\frac{\partial}{\partial\t}$ as above, where in the last equality we used the orthogonality
 property $x'\cdot\nabla_{\S^{d-1}}F(\t,x')=0$.
Hence, in the coordinates,
$|\nabla_{\S^d}f(x)|^2=|\delta F(\t,x')|^2+\rho_{\S^d}(\t)|\nabla_{\S^{d-1}}F(\t,x')|^2$.
Therefore, by \eqref{definicion integral fraccionaria}, in the coordinates we have
\begin{equation}\label{Riesz separada}
|R_{\S^d}f|^2=|\delta (-\Delta_{\S^d})^{-1/2}F|^2+
\rho_{\S^d}(\t)|\nabla_{\S^{d-1}}(-\Delta_{\S^d})^{-1/2}F|^2.
\end{equation}

\begin{proof}[Proof of Theorem \ref{Thm:Riesz Riemannian} for $M=\S^d$]
By projecting $F$ on the space of spherical harmonics, we can write
\begin{equation}\label{geografia}
F(\theta,x')=\sum_{j=0}^\infty\sum_{k=1}^{d(j)}F_{j,k}(\theta)Y_{j,k}^d(x'),\quad\hbox{for}~\theta\in(0,\pi),~x'\in\mathbb{S}^{d-1},
\end{equation}
where
$$F_{j,k}(\theta)=\int_{\mathbb{S}^{d-1}}F(\theta,x')\overline{Y_{j,k}^d(x')}\,dx'.$$
With this,
\begin{equation}\label{norma}
\|f\|_{L^p(L^2(\S^d))}=\Bigg(\int_0^\pi\Big(\sum_{j=0}^\infty\sum_{k=1}^{d(j)}|F_{j,k}(\theta)|^2\Big)^{p/2} (\sin\theta)^{d-1}d\theta\Bigg)^{1/p}.
\end{equation}
By applying \eqref{geografia} and the orthogonality of the spherical harmonics $Y^d_{j,k}$, we can write
\begin{equation}\label{coeficientes radiales theta}
\begin{aligned}
    \langle f,\varphi_{n,j,k}\rangle_{L^2(\S^d)} &= \langle F,\psi_{n,j}Y^d_{j,k}\rangle_{L^2((0,\pi)\times\S^{d-1},(\sin\theta)^{d-1}d\theta\,dx')} \\ 
    &= \int_0^\pi\int_{\S^{d-1}}\sum_{l=0}^\infty\sum_{s=1}^{d(l)}F_{l,s}(\theta)Y^d_{l,s}(x')\psi_{n,j}(\t) \overline{Y^d_{j,k}(x')}\,dx'\,(\sin\theta)^{d-1}d\theta\\
    &= \sum_{l=0}^\infty\sum_{s=1}^{d(l)}\int_0^\pi F_{l,s}(\theta)\psi_{n,j}(\t)(\sin\theta)^{d-1}d\theta\int_{\S^{d-1}}Y^d_{l,s}(x') \overline{Y^d_{j,k}(x')}\,dx' \\
    &= \int_0^\pi F_{j,k}(\theta)\psi_{n,j}(\t)(\sin\theta)^{d-1}d\theta= \langle F_{j,k},\psi_{n,j}\rangle_{L^2((0,\pi),(\sin\theta)^{d-1}d\theta)}.
\end{aligned}
\end{equation}
By using \eqref{definicion integral fraccionaria}, \eqref{proyecciones}, \eqref{coeficientes radiales theta},
\begin{align*}
    (-\Delta_{\S^d})^{-1/2}f(x) &= \sum_{j=0}^\infty\sum_{n=j}^\infty\sum_{k=1}^{d(j)}\frac{\langle F_{j,k},\psi_{n,j}\rangle_{L^2((0,\pi),(\sin\theta)^{d-1}d\theta)}}{n+\frac{d-1}{2}}\psi_{n,j}(\theta)Y^d_{j,k}(x') \\
    &= \sum_{j=0}^\infty\sum_{n=0}^\infty\sum_{k=1}^{d(j)}\frac{\langle F_{j,k},\psi_{n+j,j}\rangle_{L^2((0,\pi),(\sin\theta)^{d-1}d\theta)}}{n+\frac{2j+d-1}{2}}\psi_{n+j,j}(\theta)Y^d_{j,k}(x') \\
    &= \sum_{j=0}^\infty\sum_{k=1}^{d(j)}\left[\sum_{n=0}^\infty\frac{\langle F_{j,k},\psi_{n+j,j}\rangle_{L^2((0,\pi),(\sin\theta)^{d-1}d\theta)}}{n+\frac{2j+d-1}{2}}\psi_{n+j,j}(\theta)\right]Y^d_{j,k}(x'). \\
\end{align*}
Therefore, by \eqref{Riesz separada}, the Riesz transform can be written as
\begin{align*}
|R_{\S^d}f(x)|^2 &= \Bigg|\sum_{j=0}^\infty\sum_{k=1}^{d(j)}\left[\sum_{n=0}^\infty\frac{\langle F_{j,k},\psi_{n+j,j}\rangle_{L^2((0,\pi),(\sin\theta)^{d-1}d\theta)}}{n+\frac{2j+d-1}{2}}\delta\psi_{n+j,j}(\theta)\right]Y^d_{j,k}(x')\Bigg|^2\\
	& \quad +\rho_{\S^d}(\t)\Bigg|\sum_{j=0}^\infty\sum_{k=1}^{d(j)}\left[\sum_{n=0}^\infty\frac{\langle F_{j,k},\psi_{n+j,j}\rangle_{L^2((0,\pi),(\sin\theta)^{d-1}d\theta)}}{n+\frac{2j+d-1}{2}}\psi_{n+j,j}(\theta)\right]\nabla_{\S^{d-1}}
	Y^d_{j,k}(x')\Bigg|^2.
\end{align*}
Since
$$
\langle F_{j,k},\psi_{n+j,j}\rangle_{L^2((0,\pi),(\sin\theta)^{d-1}d\theta)}=2^{j+\frac{d-1}{2}}
\langle(\sin\t)^{-j} F_{j,k},\P_n^{(\a+j,\a+j)}\rangle_{L^2(d\mu_{\a+j,\a+j})},
$$
and $-\tilde{\Delta}_{\S^{d-1}}Y_{j,k}^d=j(j+d-2)Y_{j,k}^d$, we have
\begin{align*}
	&\int_{\S^{d-1}}|R_{\S^d}f(x)|^2\,dx' \\
	&= \sum_{j=0}^\infty\sum_{k=1}^{d(j)}\left[\left|\sum_{n=0}^\infty\frac{\langle(\sin\t)^{-j} F_{j,k},\P_n^{(\a+j,\a+j)}\rangle_{L^2(d\mu_{\a+j,\a+j})}}{n+\frac{2j+d-1}{2}} \delta\left[(\sin\t)^j\P_n^{(\a+j,\a+j)}(\t)\right]\right|^2\right. \\
	&\qquad \left.+\rho_{\S^d}(\t)\left|\sum_{n=0}^\infty\frac{\langle(\sin\t)^{-j} F_{j,k},\P_n^{(\a+j,\a+j)}\rangle_{L^2(d\mu_{\a+j,\a+j})}}{n+\frac{2j+d-1}{2}}(\sin\t)^j\P_n^{(\a+j,\a+j)}(\t)\right|^2j(j+d-2)\right]\\
	&\le2\sum_{j=0}^\infty\sum_{k=1}^{d(j)}\left[\left|\sum_{n=0}^\infty\frac{\langle(\sin\t)^{-j} F_{j,k},\P_n^{(\a+j,\a+j)}\rangle_{L^2(d\mu_{\a+j,\a+j})}}{n+\frac{2j+d-1}{2}}(\sin\t)^j\delta\P_n^{\a+j,\a+j}(\t)\right|^2\right. \\
	&\qquad +\rho_{\S^d}(\t)\left|j(\sin\t)^j\sum_{n=0}^\infty\frac{\langle(\sin\t)^{-j} F_{j,k},\P_n^{(\a+j,\a+j)}\rangle_{L^2(d\mu_{\a+j,\a+j})}}{n+\frac{2j+d-1}{2}}\P_n^{(\a+j,\a+j)}(\t)\right|^2 \\
	&\qquad \left.+\rho_{\S^d}(\t)\left|\sum_{n=0}^\infty\frac{\langle(\sin\t)^{-j} F_{j,k},\P_n^{(\a+j,\a+j)}\rangle_{L^2(d\mu_{\a+j,\a+j})}}{n+\frac{2j+d-1}{2}}(\sin\t)^j\P_n^{(\a+j,\a+j)}(\t)\right|^2j(j+d-2)\right] \\
	&\le C_d\sum_{j=0}^\infty\sum_{k=1}^{d(j)}\left[\left|(\sin\theta)^j \mathcal{R}^{\a+j,\a+j}\left((\sin\phi)^{-j}F_{j,k}\right)(\t)\right|^2+\left|j(\sin\theta)^j \mathcal{T}_{\S^d}^{\a+j,\a+j}\left((\sin\phi)^{-j}F_{j,k}\right)(\t)\right|^2\right].
\end{align*}
By Theorem \ref{Cor:Lp} and Theorem \ref{Thm:Lp angular} with
$a=b=1$ and $w=1$, from \eqref{norma}, we get
\begin{align*}
  &\|R_{\S^d}f\|_{L^p(L^2(\S^d))} \\
    &\le C\Bigg(\int_0^\pi\left[\sum_{j=0}^\infty\sum_{k=1}^{d(j)}\left|(\sin\theta)^j \mathcal{R}^{\a+j,\a+j}\left((\sin\phi)^{-j}F_{j,k}\right)(\t)\right|^2\right]^{p/2}\,(\sin\t)^{d-1}\,d\t\Bigg)^{1/p}\\
	&\quad+C\Bigg(\int_0^\pi\left[\sum_{j=0}^\infty\sum_{k=1}^{d(j)}\left|j(\sin\theta)^j \mathcal{T}_{\S^d}^{\a+j,\a+j}\left((\sin\phi)^{-j}F_{j,k}\right)(\t)\right|^2\right]^{p/2}\,(\sin\t)^{d-1}\,d\t\Bigg)^{1/p} \\
     &\leq C\Bigg(\int_0^\pi\Bigg(\sum_{j=0}^\infty \sum_{k=1}^{d(j)}|F_{j,k}(\theta)|^2\Bigg)^{p/2}(\sin\theta)^{d-1}d\theta\Bigg)^{1/p}
     = C\|f\|_{L^p(L^2(\S^{d}))}.
\end{align*}
\end{proof}

\subsection{The case of the real projective space $P_d(\Real)$}

To deal with the real projective space we just have to consider even functions on the sphere (as done at the end of Section 2 of \cite{Sherman} or in \cite[Subsection~3.2]{Ciaurri-Roncal-Stinga}, see also \cite[Chapter~III,~Section~C.II]{BGM} or \cite[p.~36]{Chavel Eigenvalues}). More precisely, a function defined on $P_d(\Real)$ can be identified with an even function on $\S^d\subset\Real^{d+1}$. Indeed, the antipodal map $s\longmapsto\pm s$ from $\S^d$ to $P_d(\Real)$ is a Riemannian
covering. The spherical harmonics in $\S^d$ satisfy $Y_{n,j}^{d+1}(-x)=(-1)^nY^{d+1}_{n,j}(x)$. Then the space $L^2_{\mathrm{e}}(\S^d)$ of even functions in $L^2(\S^d)$ can be decomposed as $L^2_{\mathrm{e}}(\S^d)=\bigoplus_{n=0}^\infty\mathcal{H}_{2n}(\S^d)$. Hence, a function on $P_d(\Real)$
is written as
$$f = \sum_{n=0}^\infty\operatorname{Proj}_{\mathcal{H}_{2n}(\S^d)}(f)= \sum_{n\geq0,\,n\,\tiny{\hbox{even}}}^\infty \operatorname{Proj}_{\mathcal{H}_{n}(\S^d)}(f)(x).$$
Thus, Theorem \ref{Thm:Riesz Riemannian} for $P_d(\Real)$ is then established as a particular case
of the result for the sphere applied to even functions.

\subsection{The case of the projective spaces $\CP$, $\HP$ and $\Ca$}

In this subsection we use the tools developed in \cite[Section~4]{Sherman}.
Let $M$ be any of the projective spaces $\CP$, $\HP$ or $\Ca$ and take $-\Delta_M$ as in the introduction, with
$\lambda_M$ as in \eqref{lambdas}. We have the orthogonal direct sum decomposition $L^2(M)=\bigoplus_{n=0}^\infty\mathcal{H}_n(M)$. Each space $\mathcal{H}_n(M)$ is finite-dimensional and corresponds to the eigenspace of $\tilde{\Delta}_M$ with respect to the eigenvalue $-n(n+m+d)$. From here it is readily seen that $\mathcal{H}_n(M)=\{f\in C^\infty(M):-\Delta_Mf=(n+\frac{m+d}{2})^2f\}$.

Now we introduce appropriate polar coordinates on $M$ by following \cite{Sherman}.
Let $\B^{d+1}$ be the unit ball in $\Real^{d+1}$ and $\omega(r):=c_\omega r^{-1}(1-r)^m$, for $0<r<1$,
with $c_\omega=\frac{\Gamma(m+d+1)}{\Gamma(d)\Gamma(m+1)}$.  According to \cite[Lemma~4.15]{Sherman} there is a bounded linear map $E:L^1(M)\to L^1(\B^{d+1},\omega(|x|)\,dx)$ such that for every $f\in L^1(M)$,
$$\int_Mf\,d\mu_M=\int_{\B^{d+1}}E(f)\omega(|x|)\,dx,$$
where $d\mu_M$ is the Riemannian measure on $M$. Each $x\in \B^{d+1}$, $x\neq0$, can be written in
polar coordinates as $x=rx'$, where $0<r=|x|<1$ and $x'\in\S^d$. Then, if we write $r=(\sin\frac{\t}{2})^2$ for $0<\theta<\pi$ and $F(\t,x')=E(f)(rx')$ we have that integration over $M$ reduces to
$$\int_Mf\,d\mu_M=c_{\omega}\int_0^\pi\int_{\S^d}F(\t,x')\,dx'(\sin\tfrac{\t}{2})^{2d-1}(\cos\tfrac{\t}{2})^{2m+1}\,d\t.$$
Let $\mathcal{F}_n(\B^{d+1})$ be the space of functions on $\B^{d+1}$ which are polynomials of degree less than or equal to $n$ in the variables $x$ and $r=|x|$. Define the set $\mathcal{H}_n(\B^{d+1},\omega)$ as the orthocompliment of the space $\mathcal{F}_{n-1}(\B^{d+1})$ in $\mathcal{F}_n(\B^{d+1})$ with respect to the inner product on $L^2(\B^{d+1},\omega(|x|)\,dx)$.
It is shown in \cite[Corollary~4.26]{Sherman} that $E(\mathcal{H}_n(M))=\mathcal{H}_n(\B^{d+1},\omega)$. With polar coordinates and the trigonometric change as above,
the eigenspaces $\mathcal{H}_n(\B^{d+1},\omega)=\mathcal{H}_n((0,\pi)\times\mathbb{S}^d,c_{\omega}d\mu_{d-1,m}\times
dx')$. These spaces are eigenspaces of the differential operator
\begin{align}\label{laplaciano polares M}
	-\Delta_M &= -\frac{\partial^2}{\partial \t^2}-\frac{(d-1-m)+(d+m)\cos\t}{\sin\t}\frac{\partial}{\partial \t}
		+\lambda_M-\frac{1}{\sin^2\frac{\t}{2}}\tilde{\Delta}_{\S^d}
\notag	= \J^{\a,\b}-\rho_M(\t)\tilde{\Delta}_{\S^d},
\end{align}
where $\a=d-1\geq1$ and $\b=m\geq0$, see \cite[Section~4,~pp.~135--136]{Sherman}. The corresponding eigenvalues are $\left(n+\frac{m+d}{2}\right)^2$, see \cite[Theorem~4.22]{Sherman}. In this way it is possible to write $L^2((0,\pi)\times\mathbb{S}^d,c_{\omega}d\mu_{\a,\b}\times dx')=\bigoplus_{n=0}^\infty\mathcal{H}_n((0,\pi)\times\mathbb{S}^d,c_{\omega}d\mu_{\a,\b}\times dx')
=\bigoplus_{n=0}^\infty\bigoplus_{j=0}^n\mathcal{H}_{n,j}((0,\pi)\times\mathbb{S}^d,c_{\omega}d\mu_{\a,\b}\times dx')$. A basis of the latter spaces is
$$\varphi^M_{n,j,k}(x)=\psi^M_{n,j}(\t)Y^{d+1}_{j,k}(x'),\quad x=(\sin\tfrac{\t}{2})^2\,x',~0<\t<\pi,~x'\in\S^d,$$
where, for $P_n^{(\a,\b)}$ as in \eqref{eq1},
$$\psi^M_{n,j}(r)=(-1)^{n-j}a^M_{n,j}(\sin\tfrac{\t}{2})^{2j}P_{n-j}^{(d-1+2j,m)}(\cos\t),$$
and the set $\{Y^{d+1}_{j,k}\}_{1\leq k\leq d(j)}$ is an orthonormal basis of the space of spherical harmonics of degree $j$, and $d(j):=(2j+d-1)\frac{(j+d-2)!}{j!(d-1)!}$. By virtue of \eqref{trig pol} and \eqref{que} the normalizing constant is
$$a^M_{n,j}=\|\varphi^M_{n,j,k}\|_{L^2((0,\pi)\times\mathbb{S}^d,c_{\omega}d\mu_{\a,\b}\times dx')}^{-1}=c_\omega^{-1/2}d_{n-j}^{\a+2j,\b}.$$
Thus, the projections can be written in these coordinates as
\begin{equation}\label{eq44}
\operatorname{Proj}_{\mathcal{H}_n(M)}(f)=\sum_{j=0}^n\sum_{k=1}^{d(j)}\langle F,\varphi_{n,j,k}^{M}\rangle_{L^2((0,\pi)\times\mathbb{S}^d,c_{\omega}d\mu_{\a,\b}\times dx')}\varphi_{n,j,k}^M.
\end{equation}

The mixed norm spaces $L^p(L^2(M))$, $1\leq p<\infty$, are defined as the set of functions $f$ on $M$ for which the norm
\begin{equation}
\label{norma M}
\|f\|_{L^p(L^2(M))}=\Bigg(c_{\omega}\int_0^{\pi}\Bigg(\int_{\S^d}|F(\t,x')|^2 \,dx'\Bigg)^{p/2}~(\sin\tfrac{\t}{2})^{2d-1}(\cos\tfrac{\t}{2})^{2m+1}\,d\t\Bigg)^{1/p}
\end{equation}
is finite. The spaces $L^p(L^2(M))$ are Banach spaces.

Now we pass to the definition of the Riesz transform on the mixed norm space $L^p(L^2(M))$. By \eqref{eq44}, in the coordinates we let
\begin{equation}\label{int frac}
(-\Delta_M)^{-1/2}f\equiv (-\Delta_M)^{-1/2}F
=\sum_{n=0}^\infty\frac{1}{n+\frac{m+d}{2}}\operatorname{Proj}_{\mathcal{H}_n(M)}(f).
\end{equation}
We have $-\Delta_M=\delta_{M}^*\delta_{M}+\lambda_M$, with
$$\delta_{M}=x'\frac{\partial }{\partial \t}+\sqrt{\rho_M(\t)}\nabla_{\mathbb{S}^d},$$
and
$$
\delta^*_{M}=-\left(\frac{\partial }{\partial \t}+\frac{2m+1}{2}\cot\tfrac{\t}{2}-\frac{2d-1}{2}\tan\tfrac{\t}{2}\right)x'+
\sqrt{\rho_M(\t)}\dive_{\mathbb{S}^d}.
$$
Analogously to the case of the sphere $\mathbb{S}^d$ in Subsection \ref{subsection:redonda}, we have
$$
\int_{M}|\nabla_{M} f(x)|^2\,dx = c_{\omega}\int_0^\pi\int_{\S^{d}}\left(|\delta F(\t,x')|^2+\rho_M(\t)|\nabla_{\S^{d}}F(\t,x')|^2\right)
	\,dx'd\mu_{d-1,m}(\t).
$$
Therefore, in the coordinates, $|\nabla_Mf(x)|^2=|\delta F(\t,x')|^2+\rho_M(\t)|\nabla_{\S^{d}}F(\t,x')|^2$.
In this way, by \eqref{int frac},
\begin{equation}
\label{Riesz separada M}
|R_Mf|^2=|\delta (-\Delta_{M})^{-1/2}F|^2+\rho_M(\t)|\nabla_{\S^{d}}(-\Delta_{M})^{-1/2}F|^2.
\end{equation}

\begin{proof}[Proof of Theorem \ref{Thm:Riesz Riemannian} for $M=\CP, \HP,\Ca$.]
Parallel to the case of the sphere, we can write
$$\|f\|_{L^p(L^2(M))}=\Bigg(c_{\omega}\int_0^\pi\Big(\sum_{j=0}^\infty\sum_{k=1}^{d(j)}|F_{j,k}(\t)|^2\Big)^{p/2} d\mu_{d-1,m}(\t)\Bigg)^{1/p},$$
with
$$F_{j,k}(\t)=\int_{\mathbb{S}^{d}}F(\t,x')\overline{Y^{d+1}_{j,k}(x')}\,dx',\quad \t\in(0,\pi).$$
Proceeding as in \eqref{coeficientes radiales theta},
$$\langle F,\varphi^M_{n,j,k}\rangle_{L^2((0,\pi)\times\mathbb{S}^d,c_{\omega}d\mu_{d-1,m}\times dx')}=c_{\omega}\langle F_{j,k},\psi^M_{n,j}\rangle_{L^2(
d\mu_{d-1,m})}.$$
Hence, by \eqref{eq44} and \eqref{int frac},
$$ (-\Delta_M)^{-1/2}f(x)= \sum_{j=0}^\infty\sum_{k=1}^{d(j)}\left[\sum_{n=0}^\infty\frac{c_{\omega}\langle F_{j,k},\psi^M_{n+j,j}\rangle_{L^2(d\mu_{d-1,m})}}{n+\frac{2j+m+d}{2}}\,\psi^M_{n+j,j}(\t)\right]Y^{d+1}_{j,k}(x').$$
From here and \eqref{Riesz separada M}, we can write
 \begin{align*}
|R_{M}f(x)|^2 &= \Bigg|\sum_{j=0}^\infty\sum_{k=1}^{d(j)}\left[\sum_{n=0}^\infty\frac{c_{\omega}\langle F_{j,k},\psi^M_{n+j,j}\rangle_{L^2(d\mu_{d-1,m})}}{n+\frac{2j+m+d}{2}}\,\delta\psi^M_{n+j,j}(\t)\right]Y^{d+1}_{j,k}(x')\Bigg|^2\\
	& \quad +\rho_M(\t)\Bigg|\sum_{j=0}^\infty\sum_{k=1}^{d(j)}\left[\sum_{n=0}^\infty\frac{c_{\omega}\langle F_{j,k},\psi^M_{n+j,j}\rangle_{L^2(d\mu_{d-1,m})}}{n+\frac{2j+m+d}{2}}\,\psi^M_{n+j,j}(\t)\right]\nabla_{\S^{d}}
	Y^{d+1}_{j,k}(x')\Bigg|^2.
\end{align*}
Recall that we are taking $\alpha=d-1, \beta=m$, so we obtain
$$
c_{\omega}\langle F_{j,k},\psi_{n+j,j}\rangle_{L^2(d\mu_{d-1,m})}=
c_{\omega}^{1/2}\langle(\sin\tfrac{\t}{2})^{-2j} F_{j,k},\P_n^{(\a+2j,\b)}\rangle_{L^2(d\mu_{\a+2j,\b})}.
$$
Hence, analogously to the case of the sphere,
\begin{align*}
	&\int_{\S^{d}}|R_{M}f(x)|^2\,dx' \\
	&\le 2\sum_{j=0}^\infty\sum_{k=1}^{d(j)}\left[\left|\sum_{n=0}^\infty\frac{c_{\omega}^{1/2}\langle(\sin\tfrac{\t}{2})^{-2j} F_{j,k},\P_n^{(\a+2j,\b)}\rangle_{L^2(d\mu_{\a+2j,\b})}}{n+\frac{2j+m+d}{2}} (\sin\tfrac{\t}{2})^{2j}\delta\P_n^{\a+2j,\b}(\t)\right|^2\right. \\
	&\qquad +\rho_M(\t)\left|j(\sin\tfrac{\t}{2})^{2j}\sum_{n=0}^\infty\frac{c_{\omega}^{1/2}\langle(\sin\tfrac{\t}{2})^{-2j} F_{j,k},\P_n^{(\a+2j,\b)}\rangle_{L^2(d\mu_{\a+2j,\b})}}{n+\frac{2j+m+d}{2}}\P_n^{(\a+2j,\b)}(\t)\right|^2 \\
	&\qquad \left.+\rho_M(\t)\left|\sum_{n=0}^\infty\frac{c_{\omega}^{1/2}\langle(\sin\tfrac{\t}{2})^{-2j} F_{j,k},\P_n^{(\a+2j,\b)}\rangle_{L^2(d\mu_{\a+2j,\b})}}{n+\frac{2j+m+d}{2}}(\sin\tfrac{\t}{2})^{2j}\P_n^{(\a+2j,\b)}(\t)\right|^2j(j+d-1)\right] \\
	&\le C\sum_{j=0}^\infty\sum_{k=1}^{d(j)}\left[\left|(\sin\tfrac{\theta}{2})^{2j} \mathcal{R}^{\a+2j,\b}\left((\sin\tfrac{\phi}{2})^{-2j}F_{j,k}\right)(\t)\right|^2+\left|j(\sin\tfrac{\theta}{2})^{2j} \mathcal{T}_M^{\a+2j,\b}\left((\sin\tfrac{\phi}{2})^{-2j}F_{j,k}\right)(\t)\right|^2\right].
\end{align*}
By Theorem \ref{Cor:Lp} and Theorem \ref{Thm:Lp angular} with $a=2$, $b=0$ and $w=1$, by using \eqref{norma M},
\begin{align*}
  &\|R_{M}f\|_{L^p(L^2(M))} \\
    &\le C\left(\int_0^\pi\left[\sum_{j=0}^\infty\sum_{k=1}^{d(j)}\left|(\sin\tfrac{\theta}{2})^{2j} \mathcal{R}^{\a+2j,\b}\left((\sin\tfrac{\phi}{2})^{-2j}F_{j,k}\right)(\t)\right|^2\right]^{p/2}\,d\mu_{\a,\b}(\t)\right)^{1/p}\\
	&\quad+C\left(\int_0^\pi\left[\sum_{j=0}^\infty\sum_{k=1}^{d(j)}\left|j(\sin\tfrac{\theta}{2})^{2j} \mathcal{T}_M^{\a+2j,\b}\left((\sin\tfrac{\phi}{2})^{-2j}F_{j,k}\right)(\t)\right|^2\right]^{p/2}\,d\mu_{\a,\b}(\t)\right)^{1/p}\\
     &\leq C\left(\int_0^\pi\left(\sum_{j=0}^\infty \sum_{k=1}^{d(j)}|F_{j,k}(\theta)|^2\right)^{p/2}\,d\mu_{\a,\b}(\t)\right)^{1/p}
     = C\|f\|_{L^p(L^2(M))}.
\end{align*}
 \end{proof}
 
\section{The extrapolation theorem and the vector-valued extensions}\label{proofs vector-valued}

In this section we show the adaptation of Rubio de Francia's extrapolation theorem in Theorem \ref{th:extra-RF} and the weighted vector-valued extensions for the Jacobi--Riesz transforms and the auxiliary operator $\mathcal{T}_M^{\a,\b}$, collected in Theorems \ref{Cor:Lp} and \ref{Thm:Lp angular}. 

For $1< p<\infty$, we denote by $A_p^{\a,\b}$ the class of $A_p$ weights on the space of homogeneous type $((0,\pi),d\mu_{\a,\b}(\t),|\cdot|)$, see \cite{Calderon}.
Namely, $A_p^{\a,\b}$ is the class of nonnegative functions $w\in L_{\mathrm{loc}}^{1}(d\mu_{\a,\b})$ for which $w^{-p'/p}\in L_{\mathrm{loc}}^{1}(d\mu_{\a,\b})$ and
$$\sup_{I\subset(0,\pi)}\left(\frac{1}{\mu_{\a,\b}(I)}\int_Iw\,d\mu_{\a,\b}\right)
\left(\frac{1}{\mu_{\a,\b}(I)}\int_Iw^{-p'/p}\,d\mu_{\a,\b}\right)^{p/p'}<\infty.$$
The Jacobi--Hardy--Littlewood maximal function is given by
$$\mathcal{M}_{\a,\b} f(\t)=\sup_{\t\in I}\frac{1}{\mu_{\a,\b}(I)}\int_{I}|f(\v)|\,
d\mu_{\a,\b}(\v),$$
where $|I|$ denotes the length of the interval $I\subset(0,\pi)$.
As a particular case of \cite[Theorem~3]{Calderon}, we have that $\mathcal{M}_{\a,\b}$ is bounded in
$L^p(w\,d\mu_{\a,\b})$ for $w\in A_p^{\a,\b}$, $1<p<\infty$, 
and it satisfies a weighted weak-$(1,1)$ estimate
(we say that $w\in A_1^{\a,\b}$ whenever $\mathcal{M}_{\a,\b}w(\t)\leq Cw(\t)$, $d\mu_{\a,\b}$-a.e.).

\begin{thm}[Extrapolation theorem]\label{th:extra-RF}
Assume that for some family of pairs of nonnegative functions $(f,g)$, for
some fixed $1<r<\infty$ and for all $w\in A_r^{\a,\b}$ we have
$$
\int_0^\pi g(\t)^rw(\t)\,d\mu_{\a,\b}(\t)\le
C\int_0^{\pi}f(\t)^rw(\t)\,d\mu_{\a,\b}(\t),
$$
with $C$ non depending on the pair $(f,g)$. Then, for all $1<p<\infty$ and all
$w\in A_p^{\a,\b}$ we have
$$
\int_0^\pi g(\t)^pw(\t)\,d\mu_{\a,\b}(\t)\le C\int_0^{\pi}
f(\t)^pw(\t)\,d\mu_{\a,\b}(\t).
$$
\end{thm}

\begin{proof}
We follow the proof given by J. Duoandikoetxea in \cite[Theorem~3.1]{Duo Extrapolation}.
The main ingredients are the factorization theorem (see \cite[Lemma~2.1]{Duo Extrapolation})
and the construction of the Rubio de Francia weights $Rf$ and $RH$ (see \cite[Lemma~2.2]{Duo Extrapolation})
in our context. The first ingredient is available here because of the general factorization
theorem proved by Rubio
de Francia \cite[Section~3]{Rubio2}. For the second ingredient we use the operator $\mathcal{M}_{\a,\b}$
to construct the weights $Rf$ and $RH$ as in Lemma 2.1 of \cite{Duo Extrapolation}.
Then the proof follows the same lines as in \cite{Duo Extrapolation}.
\end{proof}

Let
\begin{equation}
\label{ujs}
u_j(\t)=(\sin\tfrac{\t}{2})^{aj}(\cos\tfrac{\t}{2})^{bj},\quad\t\in(0,\pi),~j=0,1,\ldots.
\end{equation}

\begin{thm}[Vector-valued extension for $\mathcal{R}^{\a,\b}$]\label{Cor:Lp}
Let $\a,\b>-1/2$, $a\geq1$, $b=0$ or $b\geq1$, $1<p,r<\infty$. Let $u_j$ be as in \eqref{ujs}.
Then there is a constant $C$ such that
$$
\Big\|\Big(\sum_{j,k=0}^\infty|u_j\mathcal{R}^{\a+aj,\b+bj}(u_j^{-1}f_{j,k})|^r\Big)^{1/r}\Big\|_{L^p(w\,d\mu_{\a,\b})}\le C\Big\|\Big(\sum_{j,k=0}^\infty |f_{j,k}|^r\Big)^{1/r}\Big\|_{L^p(w\,d\mu_{\a,\b})},$$
for all $f_{j,k}\in L^p(w\,d\mu_{\a,\b}):=L^p((0,\pi),w(\t)\,d\mu_{\a,\b}(\t))$ and $w\in A_p^{\a,\b}$.
\end{thm}

\begin{thm}[Vector-valued extension for $\mathcal{T}_M^{\a,\b}$]
\label{Thm:Lp angular}
Let $\a, \b>-1/2$, $a=1$ or $a\geq2$, $b=0$ or $b\geq1$ and $1<p,r<\infty$. Let $u_j$ be as in \eqref{ujs}.
When $M=\S^d$, we also assume that $\b>0$ and $b\geq1$.
Then there exists a constant $C$ such that
$$
\Big\|\Big(\sum_{j,k=1}^\infty|ju_j\mathcal{T}_M^{\a+aj,\b+bj}(u_j^{-1}f_{j,k})|^r
\Big)^{1/r}\Big\|_{L^p(w\,d\mu_{\a,\b})}\le C\Big\|\Big(\sum_{j,k=1}^\infty |f_{j,k}|^r\Big)^{1/r}\Big\|_{L^p(w\,d\mu_{\a,\b})},
$$
for all $f_{j,k}\in L^p(w\,d\mu_{\a,\b})$ and all $w\in A_p^{\a,\b}$.
\end{thm}

It will be shown in Section \ref{Section:Kernels}
that the operators above are
Calder\'on--Zygmund operators with associated kernels
$u_j(\t)u_j(\v)\mathcal{K}^{\a+aj,\b+bj}(\t,\v)$ and $ju_j(\t)u_j(\v)T_M^{\a+aj,\b+bj}(\t,\v)$.
See Section \ref{Section:Kernels} for the precise definition of both kernels.
For them we have the following Calder\'on--Zygmund estimates,
whose proofs can be found in Subsections \ref{Section:KernelRiesz} and \ref{Sec:angular}.

\begin{thm}[Sharp estimates for the Jacobi--Riesz kernel]\label{Thm:Jacobi-Riesz kernel}
Let $\a,\b>-1/2$, $a\geq1$, $b=0$ or $b\geq1$, and $j\in\mathbb{N}_0$. 
Then
\begin{equation}
\label{eq:cresimiento}
|u_j(\t)u_j(\v)\mathcal{K}^{\a+aj,\b+bj}(\t,\v)|\le \frac{C_1}{\mu_{\a,\b}(B(\t,|\t-\v|))}, \quad \t\neq\v,
\end{equation}
and
\begin{equation}
\label{eq:suavidad}
|\nabla_{\t,\v}\big(u_j(\t)u_j(\v)\mathcal{K}^{\a+aj,\b+bj}(\t,\v)\big)|
\le \frac{C_2}{|\t-\v|\mu_{\a,\b}(B(\t,|\t-\v|))}, \quad \t\neq\v,
\end{equation}
with $C_1$ and $C_2$ independent of $j$, where $\mu_{\a,\b}(B(\t,|\t-\v|))$ is the $d\mu_{\a,\b}$-measure
of the interval with center $\t$ and radius $|\t-\v|$.
\end{thm}

\begin{thm}[Sharp estimates for the kernel of $\mathcal{T}_M^{\a,\b}$]
\label{Angular Riesz kernel}
Let $\a, \b>-1/2$, $a=1$ or $a\geq2$, $b=0$ or $b\geq1$, $j\ge1$ and $u_j$ be as in \eqref{ujs}. Then
\begin{equation}
\label{eq:cresimiento angular}
|ju_j(\t)u_j(\v)T_M^{\a+aj,\b+bj}(\t,\v)|\le \frac{C_1}{\mu_{\a,\b}(B(\t,|\t-\v|))}, \quad \t\neq\v,
\end{equation}
\begin{equation}
\label{eq:suavidad angular}
|j\nabla_{\t,\v}(u_j(\t)u_j(\v)T_M^{\a+aj,\b+bj}(\t,\v))|\le \frac{C_2}{|\t-\v|\mu_{\a,\b}(B(\t,|\t-\v|))}, \quad \t\neq\v,
\end{equation}
with $C_1$ and $C_2$ independent of $j$.
\end{thm}

\begin{proof}[Proof of Theorems \ref{Cor:Lp} and \ref{Thm:Lp angular}.] Let $\mathcal{S}_j$, $j\geq1$, be either
the operator $u_j\mathcal{R}^{\a+aj,\b+bj}(u_j^{-1}\cdot)$, or the operator $ju_j\mathcal{T}_M^{\a+aj,\b+bj}(u_j^{-1}\cdot)$, for $\a,\b>-1/2$.
By Theorems \ref{Thm:Jacobi-Riesz kernel} and  \ref{Angular Riesz kernel} and the Calder\'on--Zygmund theory
for spaces of homogeneous type (see \cite{Calderon} and \cite{RT}, also \cite{RRT}), we see that $\mathcal{S}_j$ is bounded in
$L^r(w\,d\mu_{\a,\b})$, for $1<r<\infty$ and all $w\in A_r^{\a,\b}$, uniformly in $j\geq0$.
That is, there exists a constant
$C$ independent of $j$ such that
$$
\int_0^{\pi}|\mathcal{S}_jf(\t)|^rw(\t)\,d\mu_{\a,\b}(\t)\le C \int_0^{\pi}|f(\t)|^rw(\t)\,d\mu_{\a,\b}(\t).
$$
In particular, for any sequence of functions $f_{j,k}$ in $L^r(w\,d\mu_{\a,\b})$, we have
\begin{equation}
\label{eq:inequality con sumas}
\int_0^{\pi}\sum_{j=0}^{\infty}\sum_{k=0}^{\infty}|\mathcal{S}_jf_{j,k}(\t)|^rw(\t)\,d\mu_{\a,\b}(\t)\le C \int_0^{\pi}\sum_{j=0}^{\infty}\sum_{k=0}^{\infty}|f_{j,k}(\t)|^rw(\t)\,d\mu_{\a,\b}(\t),
\end{equation}
for all $w\in A_r^{\a,\b}$. Now, in the extrapolation theorem above we make the following choices:
$$
g=\left(\sum_{j=0}^{\infty}\sum_{k=0}^{\infty}|\mathcal{S}_jf_{j,k}|^r\right)^{1/r},\qquad
 f=\left(\sum_{j=0}^{\infty}\sum_{k=0}^{\infty}|f_{j,k}|^r\right)^{1/r}.
$$
With this pair $(f,g)$, the inequality \eqref{eq:inequality con sumas} is just the hypothesis of Theorem
\ref{th:extra-RF}. Therefore, for any $1<p<\infty$ and all $w\in A_p^{\a,\b}$,
$$
\int_{0}^{\pi}\left(\sum_{j=0}^{\infty}\sum_{k=0}^{\infty}|\mathcal{S}_jf_{j,k}|^r\right)^{p/r}w(\t)\,d\mu_{\a,\b}(\t)\le
C\int_0^{\pi}\left(\sum_{j=0}^{\infty}\sum_{k=0}^{\infty}|f_{j,k}|^r\right)^{p/r}w(\t)\,d\mu_{\a,\b}(\t).
$$
\end{proof}

\section{Kernel estimates}\label{Section:Kernels}

Let us consider the
 Poisson semigroup $\P_t^{\a,\b}$ related to $\J^{\a,\b}$. This operator is initially defined in $L^2(d\mu_{\a,\b})$ as
$$\P_t^{\a,\b}f(\t)=\sum_{n=0}^\infty e^{-t\left|n+\frac{\a+\b+1}{2}\right|}c_n^{\a,\b}(f)\mathcal{P}_n^{(\a,\b)}(\theta),\quad t>0.$$
The operator semigroup $\{\P_t^{\a,\b}\}_{t>0}$ can be written as an integral operator
$$\P_t^{\a,\b}f(\t)=\int_0^{\pi}\P_t^{\a,\b}(\t,\v)f(\v)\,d\mu_{\a,\b}(\v),$$
where the Jacobi--Poisson kernel is given by
\begin{equation}\label{serie}
\P_t^{\a,\b}(\t,\v)=\sum_{n=0}^{\infty}e^{-t\left|n+\frac{\a+\b+1}{2}\right|}\mathcal{P}_n^{(\a,\b)}(\theta)\mathcal{P}_n^{(\a,\b)}(\v).
\end{equation}
We need a more explicit expression of this kernel when $\a,\b>-1/2$. Let
$$
z\equiv z(u,v,\t,\v):=u\sin\tfrac{\t}{2}\sin\tfrac{\v}{2}+v\cos\tfrac{\t}{2}\cos\tfrac{\v}{2}, \quad \t,\v\in(0,\pi),~u,v\in[-1,1].
$$
Consider on $[-1,1]$ the measure
\begin{equation*}
d\Pi_\a(u)=\frac{\Gamma(\a+1)}{\sqrt{\pi}\Gamma(\a+1/2)}(1-u^2)^{\a-1/2}\,du,
\end{equation*}
(the analogous definition for $d\Pi_\b(v)$).
The expression for the Jacobi--Poisson kernel is (see \cite{Nowak-Sjogren Calderon})
\begin{equation*}
    \P_t^{\a,\b}(\t,\v)= \frac{\Gamma(\a+\b+2)}{2^{\a+\b+1}\Gamma(\a+1)\Gamma(\b+1)}
\int_{-1}^1\int_{-1}^1\frac{\sinh \tfrac{t}{2}} {\left(\cosh\tfrac{t}{2}-1+(1-z)\right)^{\a+\b+2}}
    \,d\Pi_\a(u)\,d\Pi_\b(v).
\end{equation*}
The following identity is easy to check:
\begin{equation}\label{springer}
\int_0^\infty \P_t(\t,\v)\,dt= \frac{\Gamma(\a+\b+1)}{\pi2^{\a+\b+2}
\Gamma(\a+1/2)\Gamma(\b+1/2)}\int_{-1}^1\int_{-1}^1
\frac{(1-u^2)^{\a-1/2}(1-v^2)^{\b-1/2}} {\left(1-z\right)^{\a+\b+1}}\,du\,dv,
\end{equation}

\subsection{Kernel estimates for $\mathcal{R}^{\a,\b}$}\label{Section:KernelRiesz}

Recall the definition of Jacobi--Riesz transform $\mathcal{R}^{\a,\b}$
 given in Subsection \ref{definicion jacobiana} above.
Now, see \cite{Nowak-Sjogren Calderon}, for a compactly supported smooth
function $f$, we have
$$\mathcal{R}^{\a,\b}f(\theta)=\int_0^\infty\delta\P_t^{\a,\b}f(\theta)\,dt= \int_{0}^{\pi}\K^{\a,\b}(\t,\v)f(\varphi)\,d\mu_{\a,\b}(\varphi),$$
for $\t$ outside the support of $f$. The kernel is given by
$$\K^{\a,\b}(\t,\v)=\int_0^{\infty} \delta \mathcal{P}_t^{\a,\b}(\t,\v)\,dt.$$
Also, $\K^{\a,\b}(\t,\v)$ satisfies standard Calder\'on--Zygmund estimates on the space of homogeneous
type $((0,\pi),|\cdot|,d\mu_{\a,\b})$. Then we get
\begin{equation}\label{kernel}
\K^{\a,\b}(\t,\v)=\frac{\Gamma(\a+\b+2)}{\pi2^{\a+\b+1}\Gamma(\a+1/2)\Gamma(\b+1/2)}\int_{-1}^1
\int_{-1}^1\frac{(1-u^2)^{\a-1/2}(1-v^2)^{\b-1/2}}{(1-z)^{\a+\b+2}}\,\partial_\t (1-z)\,du\,dv,
\end{equation}
where
\begin{equation}
\label{deriv z}
\partial_{\t} (1-z)=\tfrac12\sin\tfrac{\t-\v}{2}+\tfrac{1-u}{2}\cos\tfrac{\t}{2}\sin\tfrac{\v}{2}-\tfrac{1-v}{2}\sin\tfrac{\t}{2}\cos\tfrac{\v}{2}.
\end{equation}

In the same way, the operator $u_j\mathcal{R}^{\a+aj,\b+bj}(u_j^{-1}f)(\t)$,
with $u_j$ as in \eqref{ujs},
can be defined by using \eqref{definicion obvia}
and it can be expressed as an integral operator in the Calder\'on--Zygmund sense, with associated kernel
$u_j(\t)u_j(\v)\mathcal{K}^{\a+aj,\b+bj}(\t,\v)$. Let us establish this as a lemma.

\begin{lem}\label{CZ kernel associated}
Let $\a,\b>-1/2$, $a,b\ge1$ and $u_j$ be as in \eqref{ujs}. Take $f,g\in C_c^{\infty}(0,\pi)$ having disjoint supports.
Then,
$$
\langle u_j\mathcal{R}^{\a+aj,\b+bj}(u_j^{-1}f),g\rangle_{d\mu_{\a,\b}}=\int_0^\pi\int_0^{\pi}u_j(\t)u_j(\v)\mathcal{K}^{\a+aj,\b+bj}(\t,\v)f(\v)\overline{g(\t)}\,d\mu_{\a,\b}(\v)\,d\mu_{\a,\b}(\t).
$$
\end{lem}

The proof of Lemma \ref{CZ kernel associated} follows in a standard way just by writing the left hand side as a Fourier series and checking that it coincides, after changing the order of summation and integration
(which can be done because $f$ and $g$ have disjoint supports), with the right hand side.

Now that everything is defined, we are almost ready to prove the sharp growth and smoothness estimates for the kernel $u_j(\t)u_j(\v)\mathcal{K}^{\a+aj,\b+bj}(\t,\v)$ stated in Theorem \ref{Thm:Jacobi-Riesz kernel}.

We will use the following lemma, whose proof can be found, for example, in \cite{Ciaurri-Roncal}.

\begin{lem}\label{lem:0}
Let $c>-1/2$, $0<B<A$, $\lambda>0$ and $d\ge0$. Then
$$
\int_0^1\frac{(1-s)^{c+d-1/2}}{(A-Bs)^{c+d+\lambda+1/2}}\,ds\le \frac{C(d)}{A^{c+1/2}B^{d}(A-B)^{\lambda}},
$$
where
$\displaystyle
C(d)=\begin{cases}
\frac{\Gamma(d)\Gamma(\lambda)}{\Gamma(d+\lambda)}, & d>0,\\[2pt]
C(c), &d=0.
\end{cases}
$
\end{lem}

For $\a,\b>-1/2$ we can readily check that
\begin{equation}
\label{lem:medida bolas}
\mu_{\a,\b}(B(\t,|\t-\v|))\simeq|\t-\v|(\t+\v)^{2\a+1}(\pi-\t+\pi-\v)^{2\b+1}, \qquad \t,\v\in(0,\pi).
\end{equation}
We will often use the following well known fact \cite[eq.~6.1.46]{Abra}
\begin{equation}
\label{gambas}
\frac{\Gamma(z+r)}{\Gamma(z+t)}\simeq z^{r-t}, \quad z>0,~r,t\in \Real.
\end{equation}
The next estimate is easy to see. For $\eta>0$ and $\gamma\ge1/2$, we have
\begin{equation}
\label{function h}
(1-r)^\eta r^{\gamma-1/2}\le \left(\frac{\eta}{\eta+\gamma-1/2}\right)^\eta,\quad\hbox{when}~0<r<1.
\end{equation}
Finally, we collect several trigonometric identities that will be used throughout the proofs:
\begin{align*}
1-\cos\tfrac{\t}{2}\cos\tfrac{\v}{2}&=\big(\sin\tfrac{\t-\v}{4}\big)^2+\big(\sin\tfrac{\t+\v}{4}\big)^2\simeq \t^2+\v^2\\
1-\sin\tfrac{\t}{2}\sin\tfrac{\v}{2}&=\big(\sin\tfrac{\t-\v}{4}\big)^2+\big(\cos\tfrac{\t+\v}{4}\big)^2\simeq (\pi-\t)^2+(\pi-\v)^2\\
1-\cos\tfrac{\t}{2}\cos\tfrac{\v}{2}-\sin\tfrac{\t}{2}\sin\tfrac{\v}{2}&=1-\cos\tfrac{\t-\v}{2}=2\big(\sin\tfrac{\t-\v}{4}\big)^2\simeq (\t-\v)^2.
\end{align*}

\begin{proof}[Proof of Theorem \ref{Thm:Jacobi-Riesz kernel}]
We have to prove \eqref{eq:cresimiento} and \eqref{eq:suavidad} for $j\geq1$.

\noindent\textbf{Growth estimates: proof of \eqref{eq:cresimiento}.}
By taking into account \eqref{kernel} and \eqref{deriv z},
\begin{equation*}
u_j(\t)u_j(\v)\K^{\a+aj,\b+bj}(\t,\v) =\frac{\Gamma(\a+aj+\b+bj+2)u_j(\t)u_j(\v)}{\pi\Gamma(\a+aj+1/2)\Gamma(\b+bj+1/2)2^{\a+aj+\b+bj+2}}
(J_1+J_2-J_3),
\end{equation*}
where
\begin{equation*}
J_1:=\sin\tfrac{\t-\v}{2}\int_{-1}^1\int_{-1}^1\frac{(1-u^2)^{\a+aj-1/2}(1-v^2)^{\b+bj-1/2}}{(1-z)^{\a+aj+\b+bj+2}}\,du\,dv
\end{equation*}
\begin{equation*}
J_2:=\cos\tfrac{\t}{2}\sin\tfrac{\v}{2}\int_{-1}^1\int_{-1}^1\frac{(1-u)(1-u^2)^{\a+aj-1/2}(1-v^2)^{\b+bj-1/2}}{(1-z)^{\a+aj+\b+bj+2}}\,du\,dv
\end{equation*}
\begin{equation*}
J_3:=\sin\tfrac{\t}{2}\cos\tfrac{\v}{2}\int_{-1}^1\int_{-1}^1\frac{(1-u^2)^{\a+aj-1/2}(1-v)(1-v^2)^{\b+bj-1/2}}{(1-z)^{\a+aj+\b+bj+2}}\,du\,dv.
\end{equation*}

Let us proceed with the estimates of each term.

\noindent\textbf{Estimate for $J_1$.}
Observe that
$$J_1\leq 2^{1+\a+aj+\b+bj}\sin\tfrac{\t-\v}{2}\int_0^1\int_0^1\frac{(1-u)^{\a+aj-1/2}(1-v)^{\b+bj-1/2}}
{(1-z)^{\a+aj+\b+bj+2}}\,du\,dv.$$
We apply twice Lemma \ref{lem:0}, first in the integral in $u$, taking $c=\a, d=aj, \lambda=\b+bj+3/2, A=1-v\cos\tfrac{\t}{2}\cos\tfrac{\v}{2}$ and $B=\sin\tfrac{\t}{2}\sin\tfrac{\v}{2}$, and then in the integral in $v$, taking $c=\b, d=bj, \lambda=1, A=1-\sin\frac{\t}{2}\sin\frac{\v}{2}$ and $B=\cos\tfrac{\t}{2}\cos\tfrac{\v}{2}$. We get
\begin{align*}
\Big|&\frac{\Gamma(\a+aj+\b+bj+2)u_j(\t)u_j(\v)}{\Gamma(\a+aj+1/2)\Gamma(\b+bj+1/2)2^{\a+aj+\b+bj+2}}J_1\Big|\\
&\le \frac{\Gamma(\a+aj+\b+bj+2)\Gamma(aj)\Gamma(\b+bj+3/2)\Gamma(bj)\Gamma(1)}{\Gamma(\a+aj+1/2)\Gamma(\b+bj+1/2)\Gamma(aj+\b+bj+3/2)\Gamma(bj+1)}\frac{C}{\mu_{\a,\b}(B(\t,|\t-\v|))}.
\end{align*}
This and \eqref{gambas} give \eqref{eq:cresimiento} for $J_1$.

\noindent\textbf{Estimate for $J_2$.}
As in the first step to estimate $J_1$, one passes the integration
 in $v$ from $(-1,1)$ to $(0,1)$ and uses that $1-v^2\leq 2(1-v)$ in the numerator.
Observe that $\cos\tfrac{\t}{2}\sin\tfrac{\v}{2}\le \sin\tfrac{\v}{2}\sim \tfrac{\v}{2}\le \tfrac{\t+\v}{2}$. Then, by applying Lemma \ref{lem:0} to the remaining integral in $v$ with $c=\b, d=bj, \lambda=\a+aj+3/2, A=1-u\sin\tfrac{\t}{2}\sin\tfrac{\v}{2}$ and $B=\cos\tfrac{\t}{2}\cos\tfrac{\v}{2}$ we have that $J_2$ is bounded by
\begin{multline*}
\frac{\Gamma(bj)\Gamma(\a+aj+3/2)2^{\b+bj-1/2}}{\Gamma(\a+aj+bj+3/2)}\frac{\t+\v}{(\cos\tfrac{\t}{2}\cos\tfrac{\v}{2})^{bj}}\\
\times\int_{-1}^1\frac{(1-u)(1-u^2)^{\a+aj-1/2}}{(1-u\sin\tfrac{\t}{2}\sin\tfrac{\v}{2})^{\b+1/2}(1-u\sin\tfrac{\t}{2}\sin\tfrac{\v}{2}-\cos\tfrac{\t}{2}\cos\tfrac{\v}{2})^{\a+aj+3/2}}\, du.
\end{multline*}
Now we make the change of variable $1+u=2w$, use \eqref{function h} with $\eta=1/2$ and $\gamma=\a+aj$, and apply Lemma \ref{lem:0} taking $c=\a+1/2, d=aj, \lambda=1/2, A=1+\sin\tfrac{\t}{2}\sin\tfrac{\v}{2}-\cos\tfrac{\t}{2}\cos\tfrac{\v}{2}$ and $B=2\sin\tfrac{\t}{2}\sin\tfrac{\v}{2}$. Therefore, for $j\ge1$, we get
\begin{multline*}
\Big|\frac{\Gamma(\a+aj+\b+bj+2)u_j(\t)u_j(\v)}{\Gamma(\a+aj+1/2)\Gamma(\b+bj+1/2)2^{\a+aj+\b+bj+2}}J_2\Big|\\
\le \frac{\Gamma(\a+aj+\b+bj+2)\Gamma(bj)\Gamma(\a+aj+3/2)\Gamma(aj)\Gamma(1/2)}{\Gamma(\a+aj+1/2)\Gamma(\b+bj+1/2)\Gamma(\a+aj+bj+3/2)\Gamma(aj+1/2)(\a+aj)^{1/2}}
 \frac{C}{\mu_{\a,\b}(B(\t,|\t-\v|))}.
\end{multline*}
The latter estimate and \eqref{gambas} give \eqref{eq:cresimiento}.

\noindent\textbf{Estimate for $J_3$ when $b\geq1$.}
Note that $\sin\tfrac{\t}{2}\cos\tfrac{\v}{2}\le \cos\tfrac{\v}{2}=\sin\tfrac{\pi-\v}{2}\sim \tfrac{\pi-\v}{2}\le \tfrac{\pi-\v+\pi-\t}{2}$.
We follow the same procedure as for $J_2$, but exchanging the role of $u$ and $v$ and the corresponding values for the parameters. Let us explain this. First, by applying Lemma \ref{lem:0} to the integral in $u$ with $c=\a, d=aj, \lambda=\b+bj+3/2, A=1-v\cos\tfrac{\t}{2}\cos\tfrac{\v}{2}$ and $B=\sin\tfrac{\t}{2}\sin\tfrac{\v}{2}$, $J_3$ is bounded by
\begin{multline}\label{follon}
\frac{\Gamma(aj)\Gamma(\b+bj+3/2)2^{\a+aj-1/2}}{\Gamma(aj+\b+bj+3/2)}\frac{\pi-\t+\pi-\v}{(\sin\tfrac{\t}{2}\sin\tfrac{\v}{2})^{aj}}\\
\times\int_{-1}^1\frac{(1-v)(1-v^2)^{\b+bj-1/2}}{(1-v\cos\tfrac{\t}{2}\cos\tfrac{\v}{2})^{\a+1/2}(1-v\cos\tfrac{\t}{2}\cos\tfrac{\v}{2}-\sin\tfrac{\t}{2}\sin\tfrac{\v}{2})^{\b+bj+3/2}}\, dv.
\end{multline}
Next, we make the change of variable $1+v=2w$, use \eqref{function h} with $\eta=1/2$ and $\gamma=\b+bj$,
and apply Lemma \ref{lem:0} taking $c=\b+1/2, d=bj, \lambda=1/2, A=1+\cos\tfrac{\t}{2}\cos\tfrac{\v}{2}-\sin\tfrac{\t}{2}\sin\tfrac{\v}{2}$ and $B=2\cos\tfrac{\t}{2}\cos\tfrac{\v}{2}$. We get the estimate, which is similar to that for $J_2$,
by exchanging the parameters $\a$ and $\b$.

\noindent\textbf{Estimate for $J_3$ when $b=0$.}
Here we proceed with the integral in $u$ as in the previous case, and then observe that the integral appearing in \eqref{follon}
is bounded by the sum of two terms $J_{3,-}$ and $J_{3,+}$, given by
\[
J_{3,\pm}:=\frac{C_{\b}}{(1-\cos\tfrac{\t}{2}\cos\tfrac{\v}{2})^{\a+1/2}}\int_0^1\frac{(1-v)^{\b\mp 1/2}}{(1\pm v\cos\tfrac{\t}{2}\cos\tfrac{\v}{2}-\sin\tfrac{\t}{2}\sin\tfrac{\v}{2})^{\b+3/2}}\, dv.
\]
For $J_{3,+}$, $(1-v)^{\b+1/2}\le (1-v)^{\b}$, and we apply Lemma \ref{lem:0} with $c=\b+1/2, d=0, \lambda=1/2, A=1-\sin\tfrac{\t}{2}\sin\tfrac{\v}{2}$ and $B=\cos\tfrac{\t}{2}\cos\tfrac{\v}{2}$.
Concerning $J_{3,-}$, using that $1+v\cos\frac{\t}{2}\cos\frac{\v}{2}\geq 1-\cos\frac{\t}{2}\cos\frac{\v}{2}$, since $\b-1/2>-1$,
\[
J_{3,-}\le \frac{C_\b}{(1-\cos\tfrac{\t}{2}\cos\tfrac{\v}{2})^{\a+1/2}(1-\sin\tfrac{\t}{2}\sin\tfrac{\v}{2})^{\b+1}(1-\sin\tfrac{\t}{2}\sin\tfrac{\v}{2}-\cos\tfrac{\t}{2}\cos\tfrac{\v}{2})^{1/2}}.
\]

By plugging the estimates obtained for $J_1, J_2$ and $J_3$, and taking into account \eqref{lem:medida bolas}, we get \eqref{eq:cresimiento}.

\noindent\textbf{Smoothness estimates: proof of \eqref{eq:suavidad}. Derivative in $\t$.}
We have
$$
\frac{\partial}{\partial\t}(u_j(\t)u_j(\v)\mathcal{K}^{\a+aj,\b+bj}(\t,\v))=:K_1+K_2,
$$
where, for $J_1, J_2$ and $J_3$ as above,
\begin{align*}
K_1&=\frac{\Gamma(\a+aj+\b+bj+2)u_j(\t)u_j(\v)}{\pi\Gamma(\a+aj+1/2)\Gamma(\b+bj+1/2)2^{\a+aj+\b+bj+2}}
\left(aj\frac{\cos\tfrac{\t}{2}}{\sin\tfrac{\t}{2}}-bj\frac{\sin\tfrac{\t}{2}}{\cos\tfrac{\t}{2}}\right)(J_1+J_2-J_3)\\
&=:\frac{\Gamma(\a+aj+\b+bj+2)u_j(\t)u_j(\v)}{\pi\Gamma(\a+aj+1/2)\Gamma(\b+bj+1/2)2^{\a+aj+\b+bj+2}}
\sum_{i=1}^6K_{1,i},
\end{align*}
with
\begin{align*}
K_{1,1}&:=aj\frac{\cos\tfrac{\t}{2}}{\sin\tfrac{\t}{2}}J_1, \quad K_{1,2}:=aj\frac{\cos\tfrac{\t}{2}}{\sin\tfrac{\t}{2}}J_2,\quad K_{1,3}:=aj\frac{\cos\tfrac{\t}{2}}{\sin\tfrac{\t}{2}}J_3,\\
K_{1,4}&:=-bj\frac{\sin\tfrac{\t}{2}}{\cos\tfrac{\t}{2}}J_1,\quad K_{1,5}:=-bj\frac{\sin\tfrac{\t}{2}}{\cos\tfrac{\t}{2}}J_2, \quad
K_{1,6}:=bj\frac{\sin\tfrac{\t}{2}}{\cos\tfrac{\t}{2}}J_3,
\end{align*}
and
$$
K_2=\frac{\Gamma(\a+aj+\b+bj+2)u_j(\t)u_j(\v)}{\pi\Gamma(\a+aj+1/2)\Gamma(\b+bj+1/2)2^{\a+aj+\b+bj+1}}
\frac{\partial}{\partial \t}\mathcal{K}^{\a+aj,\b+bj}(\t,\v).
$$

\noindent\textbf{Estimate for $K_1$.}
We estimate each term $K_{1,i}$, $i=1,\ldots,6$ separately. The treatment may be tedious and long, but 
it is systematic. The ideas are the same as in the growth estimates, so we only sketch the hints to follow the proofs.

Indeed, the estimate for $K_{1,1}$ follows by using first Lemma \ref{lem:0} to the integral against $v$ with $c=\b, d=bj, \lambda=\a+aj+3/2, A=1-u\sin\tfrac{\t}{2}\sin\tfrac{\v}{2}$ and $B=\cos\tfrac{\t}{2}\cos\tfrac{\v}{2}$. From here, we proceed with the change $1+u=2w$, use \eqref{function h} with $\eta=1/2$ and $\gamma=\a+aj$ and apply again Lemma \ref{lem:0} with $c=\a+1/2, d=aj-1, \lambda=3/2, A=1+\sin\tfrac{\t}{2}\sin\tfrac{\v}{2}-\cos\tfrac{\t}{2}\cos\tfrac{\v}{2}$ and $B=2\sin\tfrac{\t}{2}\sin\tfrac{\v}{2}$. We see that the case $K_{1,4}$ for $b\geq 1$ follows analogously as $K_{1,1}$, but exchanging the roles of $u$ and $v$, in the same way as it was done with $J_3$. Observe that if $b=0$, then $K_{1,4}=0$.

Similarly, for $K_{1,2}$ we use the same procedure as in $K_{1,1}$, but using \eqref{function h} with $\eta=1$ and
$\gamma=\a+aj$, and Lemma \ref{lem:0} with $\lambda=1$ in the second application of the lemma. Again, the case of $K_{1,6}$ for $b\geq1$ is analogous to $K_{1,2}$, but exchanging the role of the variables $u$ and $v$.

For $K_{1,5}$ and $b\geq1$ we can follow parallel arguments to the ones for $K_{1,3}$ in the case $b\ge1$
below, by interchanging the roles of $u$ and $v$.

As for $K_{1,3}$, we distinguish the two cases.

\noindent\textbf{Estimate for $K_{1,3}$ when $b\ge1$.} For the $K_{1,3}$ term, we use Lemma \ref{lem:0} in the integral against $u$ with $c=\a, d=aj, \lambda=\b+bj+3/2, A=1-v\cos\tfrac{\t}{2}\cos\tfrac{\v}{2}$ and $B=\sin\tfrac{\t}{2}\sin\tfrac{\v}{2}$. Then, we make the change of variable $1+v=2w$ in the integral against $v$, and use \eqref{function h} with $\eta=1$ and $\gamma=\b+bj$. Finally, we apply again Lemma \ref{lem:0} taking $c=\b, d=bj, \lambda=1, A=1+\cos\tfrac{\t}{2}\cos\tfrac{\v}{2}-\sin\tfrac{\t}{2}\sin\tfrac{\v}{2}$ and $B=2\cos\tfrac{\t}{2}\cos\tfrac{\v}{2}$.

\noindent\textbf{Estimate for $K_{1,3}$ when $b=0$.}
First, notice that $(1-u^2)^{\a+aj-1/2}\le(1-u^2)^{\a+aj-3/2}$ for $u\in (-1,1)$. Then, we use Lemma \ref{lem:0} in the integral against $u$ with $c=\a, d=aj-1, \lambda=\b+5/2, A=1-v\cos\tfrac{\t}{2}\cos\tfrac{\v}{2}$ and $B=\sin\tfrac{\t}{2}\sin\tfrac{\v}{2}$. After that, the integral arising in $v$ is
\begin{equation}\label{int}
\int_{-1}^1\frac{(1-v)(1-v^2)^{\b-1/2}}{(1-v\cos\tfrac{\t}{2}\cos\tfrac{\v}{2})^{\a+1/2}(1-v\cos\tfrac{\t}{2}\cos\tfrac{\v}{2}-\sin\tfrac{\t}{2}\sin\tfrac{\v}{2})^{\b+5/2}}\, dv.
\end{equation}
Proceeding as in the case of $J_3$ above, this integral is bounded by the sum of two terms $K_{1,3,+}$ and $K_{1,3,-}$, given by
\[
K_{1,3,\pm}:=\frac{C_{\b}}{(1-\cos\tfrac{\t}{2}\cos\tfrac{\v}{2})^{\a+1/2}}\int_0^1\frac{(1-v)^{\b\pm1/2}}{(1\mp
v\cos\tfrac{\t}{2}\cos\tfrac{\v}{2}-\sin\tfrac{\t}{2}\sin\tfrac{\v}{2})^{\b+5/2}}\, dv.
\]
For $K_{1,3,+}$, we apply Lemma \ref{lem:0} with $c=\b+1, d=0, \lambda=1, A=1-\sin\tfrac{\t}{2}\sin\tfrac{\v}{2}$ and $B=\cos\tfrac{\t}{2}\cos\tfrac{\v}{2}$.
Concerning $K_{1,3,-}$, by using that $1+v\cos\frac{\t}{2}\cos\frac{\v}{2}\geq 1-\cos\frac{\t}{2}\cos\frac{\v}{2}$, since $\b-1/2>-1$,
\[
K_{1,3,-}\le \frac{C_\b}{(1-\cos\tfrac{\t}{2}\cos\tfrac{\v}{2})^{\a+1/2}(1-\sin\tfrac{\t}{2}\sin\tfrac{\v}{2})^{\b+3/2}(1-\sin\tfrac{\t}{2}\sin\tfrac{\v}{2}-\cos\tfrac{\t}{2}\cos\tfrac{\v}{2})}.
\]

\noindent\textbf{Estimate for $K_2$.} Observe that $\partial_{\t}^2(1-z)=z/4$. Then,
\begin{align*}
\left|\frac{\partial}{\partial \t}\left(\frac{\partial_{\t}(1-z)}{(1-z)^{\a+aj+\b+bj+2}}\right)\right|&=\left|\frac{\partial_{\t}^2(1-z)}{(1-z)^{\a+aj+\b+bj+2}}-\frac{(\a+aj+\b+bj+2)(\partial_{\t}(1-z))^2}{(1-z)^{\a+aj+\b+bj+3}}\right|.
\end{align*}
From here and \eqref{kernel}, we have that
\begin{multline*}
|K_2|\le\frac{\Gamma(\a+aj+\b+bj+2)u_j(\t)u_j(\v)}{\pi\Gamma(\a+aj+1/2)\Gamma(\b+bj+1/2)2^{\a+aj+\b+bj+2}}\\
\times\int_{-1}^1\int_{-1}^1(1-u^2)^{\a+aj-1/2}(1-v^2)^{\b+bj-1/2}\sum_{i=1}^7|K_{2,i}|\,du\,dv,
\end{multline*}
where
\begin{align*}
K_{2,1}&:=\frac{z}{4(1-z)^{\a+aj+\b+bj+2}},\qquad K_{2,2}:=\frac{(\a+aj+\b+bj+2)}{(1-z)^{\a+aj+\b+bj+3}}\Big(\sin\tfrac{\t-\v}{2}\Big)^2\\
K_{2,3}&:=\frac{(\a+aj+\b+bj+2)}{(1-z)^{\a+aj+\b+bj+3}}(1-u)^2\Big(\cos\tfrac{\t}{2}\sin\tfrac{\v}{2}\Big)^2\\
K_{2,4}&:=\frac{(\a+aj+\b+bj+2)}{(1-z)^{\a+aj+\b+bj+3}}(1-v)^2\Big(\sin\tfrac{\t}{2}\cos\tfrac{\v}{2}\Big)^2\\
K_{2,5}&:=\frac{(\a+aj+\b+bj+2)}{(1-z)^{\a+aj+\b+bj+3}}(1-u)\sin\tfrac{\t-\v}{2}\cos\tfrac{\t}{2}\sin\tfrac{\v}{2}\\
K_{2,6}&:=\frac{(\a+aj+\b+bj+2)}{(1-z)^{\a+aj+\b+bj+3}}(1-v)\sin\tfrac{\t-\v}{2}\sin\tfrac{\t}{2}\cos\tfrac{\v}{2}\\
K_{2,7}&:=\frac{(\a+aj+\b+bj+2)}{(1-z)^{\a+aj+\b+bj+3}}(1-u)(1-v)\sin\tfrac{\t}{2}\cos\tfrac{\t}{2}\sin\tfrac{\v}{2}\cos\tfrac{\v}{2}.
\end{align*}
We sketch the proof for each integral containing the term $K_{2,i}$, $i=1,\ldots,7$.

For the integral with $K_{2,1}$ we observe that $z<1$ and the estimate follows in the same way as for $J_1$ above.

Concerning the integral with $K_{2,2}$, the proof is analogous to the one for $J_2$, but we take $\lambda=\b+bj+5/2$ in the first application of Lemma \ref{lem:0} and $\lambda=2$ in the second application.

We pass to $K_{2,3}$. We apply Lemma \ref{lem:0} in the integral in $v$ with $c=\b, d=bj, \lambda=\a+aj+5/2, A=1-u\sin\tfrac{\t}{2}\sin\tfrac{\v}{2}$ and $B=\cos\tfrac{\t}{2}\cos\tfrac{\v}{2}$, and  we make the change of variable $1+u=2w$ in the integral in $u$. Then, we use \eqref{function h} with $\eta=1$ and $\gamma=\a+aj$ and apply again Lemma \ref{lem:0} taking $c=\a+1, d=aj, \lambda=1, A=1+\sin\tfrac{\t}{2}\sin\tfrac{\v}{2}-\cos\tfrac{\t}{2}\cos\tfrac{\v}{2}$ and $B=2\sin\tfrac{\t}{2}\sin\tfrac{\v}{2}$.

The treatment of $K_{2,4}$ with $b\geq1$ is identical to $K_{2,3}$, but exchanging the roles of $u$ and $v$.

Next we estimate $K_{2,4}$ when $b=0$.
We use Lemma \ref{lem:0} in the integral against $u$ with $c=\a, d=aj, \lambda=\b+5/2, A=1-v\cos\tfrac{\t}{2}\cos\tfrac{\v}{2}$ and $B=\sin\tfrac{\t}{2}\sin\tfrac{\v}{2}$. Observe that the integral arising in $v$ is
the same as in \eqref{int} but with $(1-v)^2$ in place of $(1-v)$ in the numerator. This integral is then bounded by $K_{1,3,1}$ and $K_{1,3,2}$ above.

The proof of the estimate involving $K_{2,5}$ is also analogous to the one for $J_2$, but taking $\lambda=\a+aj+5/2$ in the first application of Lemma \ref{lem:0} and $\lambda=3/2$ the second time.

Concerning the integral involving $K_{2,6}$, the case $b\ge1$ works in the same way as $K_{2,5}$, but exchanging the roles of $u$ and $v$. For the case $b=0$ one uses the procedure given for $K_{2,4}$, but for the corresponding integral $K_{1,3,2}$, one needs the estimate
\[
\int_0^1\frac{(1-v)^{\b-1/2}}{(1+v\cos\tfrac{\t}{2}\cos\tfrac{\v}{2}-\sin\tfrac{\t}{2}\sin\tfrac{\v}{2})^{\b+5/2}}\, dv\le \frac{C_\b}{(1-\sin\tfrac{\t}{2}\sin\tfrac{\v}{2})^{\b+1}(1-\sin\tfrac{\t}{2}\sin\tfrac{\v}{2}-\cos\tfrac{\t}{2}\cos\tfrac{\v}{2})^{3/2}}.
\]

For the integral term with $K_{2,7}$,
 we first make the change $1+u=2w$ and use \eqref{function h} with $\eta=1/2$ and $\gamma=\a+aj$. Then, apply Lemma \ref{lem:0} taking $c=\a+1/2, d=aj, \lambda=\b+bj+2, A=1+\sin\tfrac{\t}{2}\sin\tfrac{\v}{2}-v\cos\tfrac{\t}{2}\cos\tfrac{\v}{2}$ and $B=2\sin\tfrac{\t}{2}\sin\tfrac{\v}{2}$. Later on,
 we distinguish two cases. If $b\geq1$ we proceed in the same way with the integral in $v$, that is, first make the change
 $1+v=2w$, then use \eqref{function h} and finally apply Lemma \ref{lem:0} with
  $c=\b+1/2, d=bj, \lambda=1, A=1+\cos\tfrac{\t}{2}\cos\tfrac{\v}{2}-\sin\tfrac{\t}{2}\sin\tfrac{\v}{2}$ and $B=2\cos\tfrac{\t}{2}\cos\tfrac{\v}{2}$.
If $b=0$, the integral arising in $v$ is bounded by the sum of the following two integrals:
\[
K_{2,7,\pm}:=\frac{C_{\b}}{(1-\cos\tfrac{\t}{2}\cos\tfrac{\v}{2})^{\a+1}}\int_0^1\frac{(1-v)^{\b\pm1/2}}{(1\mp v\cos\tfrac{\t}{2}\cos\tfrac{\v}{2}-\sin\tfrac{\t}{2}\sin\tfrac{\v}{2})^{\b+2}}\, dv.
\]
For the integral $K_{2,7,+}$, observe that $(1-v)^{\b+1/2}\le(1-v)^{\b}$, and we apply Lemma \ref{lem:0} with $c=\b+1/2, d=0, \lambda=1, A=1-\sin\tfrac{\t}{2}\sin\tfrac{\v}{2}$ and $B=\cos\tfrac{\t}{2}\cos\tfrac{\v}{2}$.
For $K_{2,7,-}$ we have
\[
K_{2,7,-}\le \frac{C_\b}{(1-\cos\tfrac{\t}{2}\cos\tfrac{\v}{2})^{\a+1}(1-\sin\tfrac{\t}{2}\sin\tfrac{\v}{2})^{\b+1}(1-\sin\tfrac{\t}{2}\sin\tfrac{\v}{2}-\cos\tfrac{\t}{2}\cos\tfrac{\v}{2})}.
\]

By pasting together all the estimates and taking into account \eqref{lem:medida bolas}, we get \eqref{eq:suavidad}.

\noindent\textbf{Smoothness estimates: proof of \eqref{eq:suavidad}. Derivative in $\v$.}
We have
$$
\frac{\partial}{\partial\v}(u_j(\t)u_j(\v)\mathcal{K}^{\a+aj,\b+bj}(\t,\v))=:L_1+L_2,
$$
where
\begin{equation*}
L_1=\frac{\Gamma(\a+aj+\b+bj+2)u_j(\t)u_j(\v)}{\Gamma(\a+aj+1/2)\Gamma(\b+bj+1/2)2^{\a+aj+\b+bj+2}}
\left(aj\frac{\cos\tfrac{\v}{2}}{\sin\tfrac{\v}{2}}-bj\frac{\sin\tfrac{\v}{2}}{\cos\tfrac{\v}{2}}\right)(J_1+J_2-J_3)
\end{equation*}
with $J_1, J_2$ and $J_3$ as above, and
$$
L_2=\frac{\Gamma(\a+aj+\b+bj+2)u_j(\t)u_j(\v)}{\Gamma(\a+aj+1/2)\Gamma(\b+bj+1/2)2^{\a+aj+\b+bj+2}}
\frac{\partial}{\partial \v}\left[\mathcal{K}^{\a+aj,\b+bj}(\t,\v)\right].
$$
In order to treat $L_1$, we use the symmetry of the expressions on $u$ and $v$ and $\t$ and $\v$, and the estimate \eqref{eq:suavidad} follows with analogous reasonings as for $K_1$ above.
Concerning $L_2$, we have to take into account that $|\partial^2_{\v\t}(1-z)|=\tfrac14\big|u\cos\tfrac{\t}{2}\cos\tfrac{\v}{2}+v\sin\tfrac{\t}{2}\sin\tfrac{\v}{2}\big|\le 1$ and, on the other hand, that $\partial_{\v}(1-z)=\tfrac12\sin\tfrac{\t-\v}{2}-\tfrac{1-u}{2}\sin\tfrac{\t}{2}\cos\tfrac{\v}{2}+\tfrac{1-v}{2}\cos\tfrac{\t}{2}\sin\tfrac{\v}{2}$. Then, with the same ideas as in $K_2$, the proof of the estimate \eqref{eq:suavidad} follows. We leave details to the interested reader.
\end{proof}

\subsection{Kernel estimates for $\mathcal{T}^{\a,\b}_M$}\label{Sec:angular}

Let us define the operator $\mathcal{T}_M^{\a,\b}$ in $L^2(d\mu_{\a,\b})$ as in
Subsection \ref{operador auxiliar}. By using the Poisson semigroup $\mathcal{P}_t$ in \eqref{serie}
we can see that
$$
\mathcal{T}_M^{\a,\b}f(\t)=\sqrt{\rho_M(\t)}(\mathcal{J}^{\a,\b})^{-1/2}f(\t)=
\int_0^{\pi}T_M^{\a,\b}(\t,\v)f(\v)\,d\mu_{\a,\b}(\v),\quad\hbox{in}~L^2(d\mu_{\a,\b}),
$$
where
\begin{equation}
\label{kernel T}
T_M^{\a,\b}(\t,\v)=\sqrt{\rho_M(\t)}\int_0^{\infty}\mathcal{P}_t^{\a,\b}(\t,\v)\,dt.
\end{equation}

We will prove in this subsection that $ju_j\mathcal{T}_M^{\a+aj,\b+bj}$ can be seen as 
a Calder\'on--Zygmund operator with the
kernel $ju_j(\t)u_j(\v)T_M^{\a+aj,\b+bj}(\t,\v)$, and we will deliver the proof of Theorem \ref{Angular Riesz kernel}.

\begin{lem}
Let $\a,\b>-1/2$, $a\ge1$, $b=0$ or $b\geq1$, and let $u_j$ be as in \eqref{ujs}, $j\in\mathbb{N}$.
When $M=\S^d$ we also assume that $\b>0$ and $b\geq1$.
 Then, the operator $ju_j\mathcal{T}_M^{\a+aj,\b+bj}(u_j^{-1}\cdot)$ is bounded from $L^2(d\mu_{\a,\b})$ into itself.
\end{lem}

\begin{proof}
Before starting with the estimate, we need an identity for the Jacobi polynomials. By using the identities
\[
\left(n+\frac{\a}{2}+\frac{\b}{2}\right)(1-x)P_{n-1}^{(\a+1,\b)}(x)
=(n+\a)P_{n-1}^{(\a,\b)}(x)-nP_{n}^{(\a,\b)}(x),
\]
see \cite[22.7.15]{Abra} with $n$ replaced by $n-1$, and
\[
P_{n-1}^{(\a,\b)}(x)=\frac{1}{n+\b}\left((n+\a+\b)P_{n}^{(\a,\b)}(x)-(2n+\a+\b)
P_n^{(\a-1,\b)}(x)\right),
\]
which follows from \cite[22.7.18]{Abra}, we have
\begin{equation}
\label{eq:Jacobi-Id}
\a P_{n}^{(\a,\b)}(x)=(n+\a)P_{n}^{(\a-1,\b)}(x)+\frac{n+\b}{2}(1-x)
P_{n-1}^{(\a+1,\b)}(x).
\end{equation}
With the relations $d_{n}^{\a-1,\b}=A_n d_{n}^{\a,\b}$ and $d_{n-1}^{\a+1,\b}=B_n d_n^{\a,\b}$, where $d_{n}^{\a,\b}$ is as in \eqref{que} and
\[
A_n^2=\frac{2n+\a+\b}{2n+\a+\b+1}\frac{n+\a}{n+\a+\b}, \quad B_n^2=\frac{2n+\a+\b}{2n+\a+\b+1}\frac{n+\b}{n},
\]
and the substitution $x=\cos \theta$, we conclude from \eqref{eq:Jacobi-Id} that
\begin{equation}
\label{eq:L2-1}
\frac{\alpha}{\sin \frac{\theta}{2}}\mathcal{P}_{n}^{(\a,\b)}(\t)
=\frac{n+\a}{A_n}\frac{1}{\sin\frac{\theta}{2}}
\mathcal{P}_{n}^{(\a-1,\b)}(\theta)+
\frac{n+\b}{B_n}\sin \tfrac{\theta}{2} \mathcal{P}_{n-1}^{(\a+1,\b)}(\theta).
\end{equation}

Let us start with the estimate in the case $M\neq\S^d$. The required $L^2$ estimate can be deduced from the inequality
\begin{equation}
\label{eq:L2-2}
\int_0^\pi |\a \mathcal{T}_M^{\a,\b} f|^2\, d\mu_{\a,\b}\le C \int_0^\pi |f|^2\, d\mu_{\a,\b},
\end{equation}
with $C$ a constant independent of $\a$ and $\b$.
By \eqref{eq:L2-1}, the left-hand side in \eqref{eq:L2-2} is bounded by the sum of
\[
\int_0^{\pi/2}\left|\sum_{n=0}^\infty \frac{n+\a}{\big(n+\frac{\a+\b+1}{2}\big)A_n}c_n^{\a,\b}(f)\mathcal{P}_{n}^{(\a-1,\b)}\right|^2\, d\mu_{\a-1,\b}
\]
and
\[
\int_0^{\pi/2}\left|\sum_{n=0}^\infty \frac{n+\b}{\big(n+\frac{\a+\b+1}{2}\big)B_n}c_n^{\a,\b}(f)\mathcal{P}_{n-1}^{(\a+1,\b)}\right|^2\, d\mu_{\a+1,\b}.
\]
Now, by taking into account that the sequences $\frac{n+\a}{\big(n+\frac{\a+\b+1}{2}\big)A_n}$ and $\frac{n+\b}{\big(n+\frac{\a+\b+1}{2}\big)B_n}$ are bounded by a constant independent of $\a$ and $\b$ and the orthogonality, for $\a>0$, of the systems $\{\mathcal{P}_n^{(\a-1,\b)}\}_{n\geq0}$ and $\{\mathcal{P}_n^{(\a+1,\b)}\}_{n\geq0}$, both summands are controlled by
$\sum_{n=0}^\infty |c_n^{\a,\b}(f)|^2=\int_0^\pi|f|^2\,d\mu_{\a,\b}$.

For the case $M=\S^d$, we split $[0,\pi]$ into the intervals $[0,\pi/2]$, $[\pi/2,\pi]$. In the first interval we have $\sin \theta \sim \sin \frac{\theta}{2}$ and in the second one $\sin \theta \sim \cos \frac{\theta}{2}$. Then, with the change of variable $\theta=\pi-w$ in the interval $[\pi/2,\pi]$, we can easily check that
\begin{multline*}
\int_0^{\pi}|ju_j\mathcal{T}_{\S^d}^{\a+aj,\b+bj}(u_j^{-1}f)|^2\, d\mu_{\a,\b}\\
\sim \int_0^{\pi/2}|ju_j\mathcal{T}_{\S^d}^{\a+aj,\b+bj}(u_j^{-1}f)|^2\, d\mu_{\a,\b}+
\int_0^{\pi/2}|jv_j\mathcal{T}_{\S^d}^{\b+bj,\a+aj}(v_j^{-1}g)|^2\, d\mu_{\b,\a}
\end{multline*}
with $g(w)=f(\pi-w)$, $v_j=\big(\cos\frac{\theta}{2}\big)^{aj}\big(\sin\frac{\theta}{2}\big)^{bj}$. Now, the operators
appearing in the right hand side above 
contain the factor $1/\sin \frac{\theta}{2}$ instead of $1/\sin \theta$. So, enlarging the intervals of integration, the boundedness of both integrals is reduced to the previous case.
\end{proof}

Proceeding as in Lemma \ref{CZ kernel associated}, it is easily seen that 
$ju_j\mathcal{T}_M^{\a+aj,\b+bj}(u_j^{-1}f)$ are Calder\'on--Zygmund operators with kernel
$ju_j(\t)u_j(\v)T_M^{\a+aj,\b+bj}(\t,\v)$.

\begin{proof}[Proof of Theorem \ref{Angular Riesz kernel}]
We only need to consider the case $\sqrt{\rho_M(\t)}=\frac{1}{\sin\frac{\t}{2}}$, that is, $M\neq \S^d$.
If $M=\S^d$ we observe that $(\sqrt{\rho_{\S^d}(\t)})^{-1}=\sin\t=2\sin\tfrac{\t}{2}\cos\tfrac{\t}{2}$,  which introduces a singularity in $\t=0$ and $\t=\pi$. In fact, when $\t\in (0,\pi/2)$, $(\sqrt{\rho_{\S^d}(\t)})^{-1}$ behaves as ${\sin\tfrac{\t}{2}}$, and the operator $\mathcal{T}_{\S^d}^{\a,\b}$ is treated exactly in the same way as $\mathcal{T}_M^{\a,\b}f$, $M\neq\S^d$. If $\t\in (\pi/2,\pi)$, $(\sqrt{\rho_{\S^d}(\t)})^{-1}$
behaves as ${\cos\tfrac{\t}{2}}$, and the proofs are easily adapted in order to cancel this singularity.

Then, taking into account the expression for the kernel \eqref{kernel T} and the identity
\eqref{springer}, we have
\begin{multline*}
ju_j(\t)u_j(\v)T_M^{\a+aj,\b+bj}(\t,\v) =\frac{ju_j(\t)u_j(\v)\Gamma(\a+aj+\b+bj+1)}{\pi\Gamma(\a+aj+1/2)\Gamma(\b+bj+1/2)2^{\a+aj+\b+bj+2}\sin\tfrac{\t}{2}}\\
\times\int_{-1}^1\int_{-1}^1\frac{(1-u^2)^{\a+aj-1/2}(1-v^2)^{\b+bj-1/2}}{(1-z)^{\a+aj+\b+bj+1}}\,du\,dv.
\end{multline*}
The ideas to get the estimates \eqref{eq:cresimiento angular} and \eqref{eq:suavidad angular} are
completely analogous to the ones used for the Jacobi--Riesz transform estimates in the proof of Theorem \ref{Thm:Jacobi-Riesz kernel}. In this way, we apply Lemma \ref{lem:0} and \eqref{function h} with suitable choices in every single case that arises. Further details are omitted.
\end{proof}



\end{document}